%% file: cvx_aff_cmt.tex
\documentclass[12pt]{article}

\input{pre}

\begin{document}

\author{\normalsize  Rory B. B. Lucyshyn-Wright\thanks{The author gratefully acknowledges financial support in the form of an AARMS Postdoctoral Fellowship, a Mount Allison University  Research Stipend, and, earlier, an NSERC Postdoctoral Fellowship.}\let\thefootnote\relax
\\
\small Mount Allison University, Sackville, New Brunswick, Canada}

\title{\large \textbf{Convex spaces, affine spaces, and \\commutants for algebraic theories}}

\date{}

\maketitle

\abstract{Certain axiomatic notions of \textit{affine space} over a ring and \textit{convex space} over a preordered ring are examples of the notion of $\T$-algebra for an algebraic theory $\T$ in the sense of Lawvere.  Herein we study the notion of \textit{commutant} for Lawvere theories that was defined by Wraith and generalizes the notion of \textit{centralizer clone}.  We focus on the Lawvere theory of \textit{left $R$-affine spaces} for a ring or rig $R$, proving that this theory can be described as a commutant of the theory of pointed right $R$-modules.  Further, we show that for a wide class of rigs $R$ that includes all rings, these theories are commutants of one another in the full finitary theory of $R$ in the category of sets.  We define \textit{left $R$-convex spaces} for a preordered ring $R$ as left affine spaces over the positive part $R_+$ of $R$.  We show that for any \textit{firmly archimedean} preordered algebra $R$ over the dyadic rationals, the theories of left $R$-convex spaces and pointed right $R_+$-modules are commutants of one another within the full finitary theory of $R_+$ in the category of sets.  Applied to the ring of real numbers $\RR$, this result shows that the connection between convex spaces and pointed $\RR_+$-modules that is implicit in the integral representation of probability measures is a perfect `duality' of algebraic theories.}

\section{Introduction} \label{sec:intro}

In 1963, Lawvere \cite{Law:PhD} introduced an elegant approach to Birkhoff's universal algebra through category theory.  Therein, an \textit{algebraic theory} or \textit{Lawvere theory} is by definition a category $\T$ with a denumerable set of objects $T^0,T^1,T^2,...$ in which $T^n$ is an $n$-th power of the object $T = T^1$, and a \textit{$\T$-algebra} is a functor $A:\T \rightarrow \Set$ that is valued in the category of sets and preserves finite powers.  We call $\ca{A} = A(T)$ the \textit{carrier} of $A$, and we say that $A$ is \textit{normal} if $A$ sends the powers $T^n$ in $\T$ to the canonical $n$-th powers $\ca{A}^n$ in $\Set$ \pbref{par:talgs}.  $\T$-algebras and the natural transformations between them constitute a category $\Alg{\T}$ with an equivalent full subcategory consisting of all normal $\T$-algebras \pbref{par:cat_talgs}.  The morphisms $\omega:T^n \rightarrow T$ in a Lawvere theory $\T$ may be called \textit{abstract operations}, and the mappings $A(\omega):\ca{A}^n \rightarrow \ca{A}$ associated to these by a given $\T$-algebra $A$ are then called \textit{concrete operations}.  For convenience, we can take the objects $T^n$ of $\T$ to be just the finite cardinals $n$ to which they correspond bijectively.

By a \textit{(normal) faithful representation} of a Lawvere theory $\T$ we mean a normal $\T$-algebra $R:\T \rightarrow \Set$ that is faithful as a functor.  Writing simply $R$ for the carrier of $R$, such a faithful representation presents $\T$ as a subtheory $\T \hookrightarrow \Set_R$ of a larger theory $\Set_R$ called the \textit{full finitary theory of $R$ in $\Set$}, consisting of \textit{all} the mappings between the $n$-th powers $R^n$ of the set $R$.  Such subtheories are essentially the \textit{concrete clones} that appear in universal algebra, as contrasted with the (a priori) more general \textit{abstract clones} of Hall, which correspond to arbitrary Lawvere theories.

One of the chief objectives of this paper is to study a phenomenon sometimes exhibited by a faithfully represented Lawvere theory $\T \hookrightarrow \Set_R$, wherein the set $R$ carries the structure of an $\sS$-algebra for some other Lawvere theory $\sS$ and the mappings $R^n \rightarrow R^m$ that lie within the subtheory $\T \hookrightarrow \Set_R$ are precisely those that are $\sS$-\textit{homomorphisms} with respect to the induced $\sS$-algebra structures on $R^n$ and $R^m$, so that
$$\T(n,m) \cong \Alg{\sS}(R^n,R^m)\;.$$
We will in fact encounter situations in which, moreover, the $\sS$-algebra structure on $R$ is also a faithful representation of $\sS$ with the same property, such that the subtheory $\sS \hookrightarrow \Set_R$ consists of exactly those mappings $R^n \rightarrow R^m$ that are $\T$-homomorphisms.  In symbols
$$\sS(n,m) \cong \Alg{\T}(R^n,R^m)\;.$$
It is precisely this curious `duality' of certain pairs of theories $\T$ and $\sS$ that we seek to understand.

As a first example, let us consider the theories $\T$ and $\sS$ of left and right $R$-modules, respectively, for a given ring\footnote{Throughout this paper, we use the term \textit{ring} to mean \textit{unital ring}.  A similar remark applies to the notion of \textit{rig} or \textit{semiring} employed herein, whose definition we recall in \bref{exa:law_th_rmods}.} $R$ (or even just a \textit{rig} or \textit{semiring}, \bref{exa:law_th_rmods}).  Concretely, $\T$ is the category $\Mat_R$  whose objects are the natural numbers $n$ and whose morphisms $n \rightarrow m$ are $m \times n$-matrices with entries in $R$, with composition given by matrix multiplication.  The category of normal $\T$-algebras is isomorphic to the category $\Mod{R}$ of left $R$-modules \pbref{exa:law_th_rmods}.  Similarly, $\sS$-algebras for $\sS = \Mat_{R^\op}$ are right $R$-modules.  Given a left $R$-module $A$, the corresponding normal $\T$-algebra $\T \rightarrow \Set$ is given on objects by $n \mapsto A^n$ and associates to each $m \times n$-matrix $w \in R^{m \times n}$ the mapping  $A^n \rightarrow A^m$ that sends a column vector $x \in A^n$ to the matrix product $wx \in A^m$.  A normal $\T$-algebra is uniquely determined by its carrier and its values on morphisms of the form $w:n \rightarrow 1$ in $\T$, i.e. on row vectors $w \in R^{1 \times n}$, for which the associated maps
\begin{equation}\label{eq:linear_comb_ops}A^n \rightarrow A\;,\;\;\;\;\;x \mapsto wx = \sum_{i = 1}^nw_ix_i\end{equation}
implement the taking of left $R$-linear combinations.  In particular, $R$ itself is a left $R$-module and so determines a normal $\T$-algebra $R:\T \rightarrow \Set$ that is in fact a faithful representation.  Similarly, $R$ is a right $R$-module, so we have faithful representations
$$\T \hookrightarrow \Set_R\;\;\;\;\;\;\;\;\sS \hookrightarrow \Set_R\;.$$
Thus viewing $\T$ as a subtheory of $\Set_R$, we find that the mappings $R^n \rightarrow R^m$ that lie in $\T$ are precisely the \textit{right} $R$-linear maps (i.e. the $\sS$-homomorphisms) whereas the mappings $R^n \rightarrow R^m$ in $\sS$ are precisely the \textit{left} $R$-linear maps (i.e. the $\T$-homomorphisms) \pbref{thm:th_lr_rmods_mutual_cmtnts}.

This peculiar duality of pairs of theories $\T,\sS$ can be understood through the notion of \textit{commutant} for Lawvere theories that was briefly introduced by Wraith in his lecture notes on algebraic theories \cite{Wra:AlgTh} but was not studied to any substantial extent therein.  Given a set $R$, a pair of mappings $\mu:R^n \rightarrow R^m$ and $\nu:R^{n'} \rightarrow R^{m'}$ is said to \textit{commute} if the associated mappings
$$\mu * \nu = \left(R^{n \times n'} \cong (R^n)^{n'} \xrightarrow{\mu^{n'}} (R^m)^{n'} \cong (R^{n'})^m \xrightarrow{\nu^m} (R^{m'})^m \cong R^{m \times m'}\right)$$
$$\mu \stt \nu = \left(R^{n \times n'} \cong (R^{n'})^n \xrightarrow{\nu^n} (R^{m'})^n \cong (R^n)^{m'} \xrightarrow{\mu^{m'}} (R^{m})^{m'} \cong R^{m \times m'}\right)$$
are equal.  Given a subtheory $\T \hookrightarrow \Set_R$, the \textit{commutant} of $\T$ in $\Set_R$ is, by definition, the subtheory $\T^\perp \hookrightarrow \Set_R$ consisting of those mappings $\mu:R^n \rightarrow R^m$ that commute with every mapping $\nu:R^{n'} \rightarrow R^{m'}$ in $\T$.  A subtheory $\T \hookrightarrow \Set_R$ is equivalently a theory $\T$ admitting a faithful representation with carrier $R$, and the key observation is now that a mapping $\mu:R^n \rightarrow R^n$ lies in the commutant $\T^\perp$ if and only if $\mu$ is a $\T$-homomorphism \pbref{thm:cmtnt_full_th_alg}; i.e.,
$$\T^\perp(n,m) = \Alg{\T}(R^n,R^m)\;.$$
Hence the `duality' observed above in pairs of faithfully represented theories $\T,\sS$ is equivalently the statement that $\sS$ and $\T$ are commutants of one another within the full finitary theory $\Set_R$ of a set $R$; in symbols,
$$\T \cong \sS^\perp\;\;\;\;\;\;\;\T^\perp \cong \sS\;.$$
In particular, given a ring or rig $R$, the theories $\Mat_R$ and $\Mat_{R^\op}$ of left and right $R$-modules, respectively, are commutants of one another within $\Set_R$:
$$\Mat_R \cong (\Mat_{R^\op})^\perp\;\;\;\;\;\;(\Mat_R)^\perp \cong \Mat_{R^\op}\;.$$

Wraith's notion of commutant applies not only to subtheories $\T$ of the full finitary theory $\Set_R$ of a set $R$ but also to subtheories $\T \hookrightarrow \U$ of an arbitrary Lawvere theory $\U$.  Indeed, in analogy with the above one can again define the notion of commutation of morphisms in $\U$ \pbref{def:kps_cmt}, and the \textit{commutant}
$$\T^\perp \hookrightarrow \U$$
of $\T$ in $\U$ is then defined in the analogous way \pbref{def:cmtnt}.  More generally, we can define the commutant $\T^\perp_A \hookrightarrow \U$ of a morphism of Lawvere theories $A:\T \rightarrow \U$ as the commutant of its image.  The commutant is then characterized by a universal property, namely that a morphism of theories $B:\sS \rightarrow \U$ factors uniquely through the commutant $\T^\perp_A \hookrightarrow \U$ if and only if $A$ \textit{commutes} with $B$ in a suitable sense \pbref{def:cmtn_mor_th}.  Defining a \textit{theory over $\U$} as a theory $\T$ equipped with a morphism $A:\T \rightarrow \U$, it follows that the operation $(-)^\perp$ on theories over $\U$ gives rise to an adjunction between the category of theories over $\U$ and its opposite \pbref{thm:cmtnt_adjn}, and this adjunction restricts to a Galois connection on subtheories of $\U$.  Since a normal $\T$-algebra $R:\T \rightarrow \Set$ is equivalently described as a morphism $R:\T \rightarrow \Set_R$ into the full finitary theory of its carrier $R$, we recover the commutant of a faithful representation as a special case.  In particular, when $\T$ is the subtheory of $\Set_R$ generated by a specified family of finitary operations on a given set $R$, we recover the notion of \textit{centralizer clone} that has been studied to some extent in the literature on universal algebra.  For example, the paper \cite{TrnSich} characterizes those abstract clones or Lawvere theories $\T$ for which there exists a set $R$ equipped with a family of operations whose centralizer clone is isomorphic to $\T$.

In addition to a general study of the notion of commutant for Lawvere theories\footnote{A further recent preprint \cite{Lu:Cmtnts} by the author of the present paper was made available subsequent to the initial version of the present paper and treats the general theory of commutants for $\V$-\textit{enriched} algebraic theories for a system of arities $\J \hookrightarrow \V$.  The development is much simpler in the present $\Set$-based case, and the present paper is principally concerned with specific examples.}, the present paper comprises an in-depth study of certain specific examples of commutants.  In particular, we prove several theorems concerning the theory of \textit{$R$-affine spaces} for a ring or rig $R$ and, in particular, \textit{$R$-convex spaces} for a preordered ring $R$.  Several authors have studied axiomatic notions of affine space over a ring or rig, and the generality afforded by the use of a mere rig permits the consideration of the notion of convex space as a special case; for example, see \cite{ParRoeh} and the references there.  Whereas a (left) $R$-module $A$ is a set equipped with operations \eqref{eq:linear_comb_ops} that permit the taking of linear combinations, a \textit{(left) $R$-affine space} or \textit{(left) $R$-affine module} is a set equipped with operations \eqref{eq:linear_comb_ops} that permit the taking of \textit{affine combinations}, i.e. those linear combinations $\sum_{i = 1}^nw_ix_i$ whose coefficients $w_i$ sum to $1$.  More precisely, a left $R$-affine space is by definition a normal $\T$-algebra for a certain subtheory 
$$\T = \Mat_R^\aff \hookrightarrow \Mat_R$$
of the category of $R$-matrices, namely the subtheory consisting of all matrices in which each row sums to $1$.  This way of defining the notion of $R$-affine space was given in \cite{Law:ProbsAlgTh}.  Letting $\RR_+$ denote the rig of non-negative reals, $\RR_+$-affine spaces are usually called \textit{convex spaces}, and $\RR_+$-affine combinations are called \textit{convex combinations}.  This way of defining convex spaces was given in \cite{Meng}.

We pursue answers to the following questions:
\begin{enumerate}
\item Does $\Mat_R^\aff$ arise as a commutant of some theory over the full finitary theory $\Set_R$ of $R$ in $\Set$?
\item What is the commutant of $\Mat_R^\aff$ in $\Set_R$?
\end{enumerate}
We answer 1 in the affirmative for every rig $R$.  Indeed, defining a \textit{pointed right $R$-module} as a right $R$-module $M$ equipped with a chosen element $* \in M$, the category of pointed right $R$-modules is isomorphic to the category of $\T$-algebras for a Lawvere theory $\T = \Mat_{R^\op}^*$.  The pointed right $R$-module $(R,1)$ determines a morphism $\Mat_{R^\op}^* \rightarrow \Set_R$ by which $\Mat_{R^\op}^*$ can be considered as a theory over $\Set_R$, though not in general a subtheory, and we show in \bref{thm:taffr_as_cmtnt} that
$$\Mat_R^\aff \cong (\Mat_{R^\op}^*)^\perp$$
as theories over $\Set_R$.  Hence
$$
\begin{minipage}{4.2in}
\textit{the Lawvere theory of left $R$-affine spaces is the commutant of the theory of pointed right $R$-modules when both are considered as theories over the full finitary theory of $R$ in $\Set$.}
\end{minipage}
$$
Consequently $\Mat_R^\aff$ is its own double-commutant $(\Mat_R^\aff)^{\perp\perp}$ over $\Set_R$, so we say that $\Mat_R^\aff$ is a \textit{saturated} subtheory of $\Set_R$.

It is illustrative to note that in the case of the rig $\RR_+$ one finds here a connection to the Kakutani-Markov-Riesz representation theorem, since for each finite cardinal $n$ the resulting bijection $\Mat_{\RR_+}^\aff(n,1) \cong (\Mat_{\RR_+}^*)^\perp(n,1)$ is the correspondence between probability measures on the finite set $n$ (on the left-hand side) and $1$-preserving $\RR_+$-linear functionals $\RR^n_+ \rightarrow \RR_+$ (on the right).

With regard to question 2 it is natural to ask also whether $\Mat^*_{R^\op}$ is the commutant of $\Mat_R^\aff$ over $\Set_R$.  When $R$ is a \textit{ring} we show that this is indeed the case \pbref{thm:cmtnt_of_th_of_raff_sp}, so that
\begin{equation}\label{eq:matraffperp_is_matropstar}(\Mat_R^\aff)^\perp \cong \Mat_{R^\op}^*\end{equation}
over $\Set_R$.  Hence
$$
\begin{minipage}{4.2in}
\textit{if $R$ is a ring, then the theories of left $R$-affine spaces and pointed right $R$-modules are commutants of one another within the full finitary theory of $R$ in $\Set$.}
\end{minipage}
$$
However for arbitrary \textit{rigs} this is no longer true.  For example, when $R$ is the two-element rig $2 = (2,\vee,0,\wedge,1)$, we show that (i) $2$-modules are equivalently \textit{(bounded) join semilattices} \pbref{exa:slats}, (ii) $2$-affine spaces are \textit{unbounded join semilattices} (i.e., idempotent commutative semigroups, \bref{exa:unb_jslats}), and (iii) the commutant in $\Set_2$ of the theory of unbounded join semilattices is the theory of join semilattices with a \textit{top} element \pbref{thm:th_uslats_slatswtop_cmtnts}.

Nevertheless, we show that \eqref{eq:matraffperp_is_matropstar} \textit{does} hold for many rigs other than rings.  In particular, we show that it holds for the rig $\RR_+$ of non-negative reals \pbref{exa:firmly_arch_pralgs_over_d}, so that
\begin{equation}\label{eq:cmtnt_convex_sp}
\begin{minipage}{4.5in}
\textit{the theory of convex spaces (over $\RR$) and the theory of pointed right $\RR_+$-modules are commutants of one another within the full finitary theory of $\RR_+$ in $\Set$.}
\end{minipage}
\end{equation}
This result shows that the connection between convex spaces and pointed $\RR_+$-modules that is implicit in the integral representation of probability measures is in fact a perfect `duality' of algebraic theories.  Indeed, one of the purposes of the present paper is to provide an algebraic basis for a study of measure and distribution monads canonically determined by such dualities in the enriched context \cite{Lu:CT2015,Lu:FDistn}.

In order to generalize this result, we study affine spaces over rigs of the form 
$$R_+ = \{r \in R \;|\; r \gt 0 \}$$ 
where $R$ is a \textit{preordered ring} \pbref{par:pr_rings}.  Preordered and partially ordered rings have been studied at various levels of generality in the literature on ordered algebra, and they can be defined equivalently as rings $R$ equipped with an arbitrary subrig $R_+ \hookrightarrow R$.  The rigs that occur as the \textit{positive part} $R_+$ of some preordered ring $R$ are precisely the \textit{additively cancellative rigs} \pbref{par:pr_rings}.

Given a preordered ring $R$, we call left $R_+$-affine spaces \textit{left $R$-convex spaces} or \textit{left $R$-convex modules}.  In \bref{thm:cvx_cmtnt_charn} we establish a characterization of the class of all preordered rings $R$ for which the evident analogue of \eqref{eq:cmtnt_convex_sp} holds, and we then proceed to develop sufficient conditions that entail that a preordered ring $R$ belongs to this class, as we now outline.

Whereas the \textit{archimedean property} for totally ordered fields $R$ can be expressed in several equivalent ways, certain of these statements become inequivalent when one passes to arbitrary preordered rings $R$.  In particular, we define the notion of \textit{firmly archimedean} preordered ring \pbref{def:firmly_arch}, noting that a nonzero totally ordered ring is firmly archimedean if and only if it is archimedean.  Given any integer $d > 1$, we prove that if $R$ is a firmly archimedean preordered ring and $d$ is invertible in the rig $R_+$, then the relevant analogue of \eqref{eq:cmtnt_convex_sp} holds \pbref{thm:suff_conds}.  But $d$ is invertible in $R_+$ if and only if there exists a (necessarily unique) morphism of preordered rings from the ring of \textit{$d$-adic fractions} $\ZZ[\frac{1}{d}]$ into $R$ (\bref{par:dyadic}, \bref{thm:preord_dalg}), so this result can be stated as follows:
$$
\begin{minipage}{4.5in}
\textit{Let $R$ be a firmly archimedean preordered algebra over $\ZZ[\frac{1}{d}]$, for some integer $d > 1$.  Then the Lawvere theory of left $R$-convex spaces and the Lawvere theory of pointed right $R_+$-modules are commutants of one another in the full finitary theory of $R_+$ in $\Set$.}
\end{minipage}
$$
In particular, this applies to (i) the ring of real numbers $R = \RR$, (ii) the ring of \textit{dyadic rationals} $R = \ZZ[\frac{1}{2}]$, (iii) the ring $R$ of all bounded real-valued functions on a set, or any sub-$\ZZ[\frac{1}{d}]$-algebra thereof, and in particular (iv) the ring $R = C(X)$ of all continuous functions on a compact space $X$.

We begin in \S \bref{sec:lth} with a survey of basic material concerning Lawvere theories, and we discuss several examples of Lawvere theories for use in the sequel.  In \S \bref{sec:aff_cvx} we define the Lawvere theory of left $R$-affine spaces for a rig $R$ and the Lawvere theory of left $R$-convex spaces for a preordered ring $R$.  In \S \bref{sec:cmtn} we provide a self-contained treatment of the notion of commutation of morphisms in a Lawvere theory $\T$ by studying in detail the \textit{first and second Kronecker products} of morphisms in $\T$ \pbref{def:kps_cmt}.  Noting that these Kronecker products in $\T$ depend on a choice of binary product projections in the category of finite cardinals \pbref{par:binary_prods_in_fincard}, we show that one specific such choice enables a rigourous proof that the first Kronecker product of morphisms in the category $\Mat_R$ of matrices over a rig $R$ is the classical Kronecker product of matrices \pbref{exa:kp_matr}, as we are not aware of any statement or proof of this in the literature.  In \S \bref{sec:cmtnt} we study the notion of commutant of a morphism of Lawvere theories $A:\T \rightarrow \U$, proving that it can be defined equivalently as the full finitary theory of $A$ in the category of $\T$-algebras in $\U$ \pbref{thm:cmtnt_full_th_alg}, and we treat the example of the theory of left $R$-modules for a rig $R$ \pbref{thm:cmtnt_lth_rmods}.  In \S \bref{sec:sat_bal} we show that the passage from a theory $\T$ over $\U$ to its commutant is a `self-adjoint' contravariant functor \pbref{thm:cmtnt_adjn}, and we study the notions of \textit{saturated} and \textit{balanced} subtheory \pbref{def:sat_bal}.  In \S \bref{sec:th_affsp_as_cmtnt}, \bref{sec:cmt_th_unb_slats}, \bref{sec:cmtnt_of_th_cvx_sp}, \bref{sec:cmtnt_th_cvx_sp} we derive our main results concerning the theories of $R$-affine and $R$-convex spaces and their commutants (\bref{thm:taffr_as_cmtnt}, \bref{thm:th_uslats_slatswtop_cmtnts}, \bref{thm:cmtnt_of_th_of_raff_sp}, \bref{thm:cvx_cmtnt_charn}, \bref{thm:suff_conds_for_aff_ext_prop}, \bref{thm:suff_conds}).

\begin{Acknowledgement}  The author thanks the anonymous referee for helpful suggestions and remarks.  In the original version of this paper, Theorems \bref{thm:suff_conds_for_aff_ext_prop} and \bref{thm:suff_conds} were formulated for preordered algebras $R$ over the dyadic rationals $\ZZ[\frac{1}{2}]$.  However, the referee supplied an argument to the effect that equation \eqref{eq:common_diff} still holds when one replaces $\ZZ[\frac{1}{2}]$ with the ring of $d$-adic fractions $\ZZ[\frac{1}{d}]$ for any integer $d > 1$, thus showing that Theorems \bref{thm:suff_conds_for_aff_ext_prop} and \bref{thm:suff_conds} apply to the broader classes of preordered algebras for which they are now formulated herein.
\end{Acknowledgement}

\section{Lawvere theories, their algebras, and several examples}\label{sec:lth}

\begin{ParSub}[\textbf{Lawvere theories}]\label{par:lth}
A \textbf{Lawvere theory} is a small category $\T$ equipped with an identity-on-objects functor $\tau:\FinCard^\op \rightarrow \T$ that preserves finite powers, where $\FinCard$ is the full subcategory of $\Set$ consisting of the finite cardinals.  We may identify finite cardinals with natural numbers, so that $\ob\T = \NN$ is the set of all natural numbers.  When we want to emphasize that a natural number $n$ is to be treated as an object of $\T$, we will sometimes denote it by $\tau(n)$.

Since $\FinCard$ has finite copowers, its opposite $\FinCard^\op$ has finite powers and hence every Lawvere theory $\T$ has finite powers, furnished by $\tau$.  In particular, each finite cardinal $n$ is an $n$-th copower of $1$ in $\FinCard$, so $n$ is an $n$-th \textit{power} of $1$ in $\T$.  In symbols, $n = \tau(n) = \tau(1)^n$ in $\T$.  For ease of notation we will sometimes write $T = \tau(1)$ and correspondingly write $T^n$ for the object $n$ of $\T$.  Choosing designated $n$-th copower cocones $(\iota_i:1 \rightarrow n)_{i = 1}^n$ in $\FinCard$ in the evident way, we thus obtain designated $n$-th power cones $(\pi_i = \tau(\iota_i):T^n \rightarrow T)$ in $\T$.  Moreover, $\tau$ can then be characterized as the functor $T^{(-)}:\FinCard^\op \rightarrow \T$ that is given on objects by $n \mapsto T^ n$ and sends each mapping $f:m \rightarrow n$ in $\FinCard$ to the induced morphism $T^f:T^n \rightarrow T^m$.

We say that an object $C$ of a category $\C$ has \textbf{designated finite powers} if it is equipped with a specified choice of $n$-th power $C^n$ in $\C$ for each $n \in \NN$.  We say that these designated finite powers of $C$ are \textbf{standard} if $C^1 = C$, with the identity morphism $1_C$ as the designated projection $\pi_1:C^1 \rightarrow C$.  For example, the designated $n$-th powers of $\tau(1)$ in a Lawvere theory $(\T,\tau)$ are standard, since the designated morphism $\iota_1:1 \rightarrow 1$ in $\FinCard$ is necessarily the identity.  It shall be convenient to fix a choice of standard designated finite powers $\tau(m)^n$ of each of the objects $\tau(m) = m$ of $\T$, and we may assume that this choice of powers extends the basic choice $\tau(1)^n = n$ in the case that $m = 1$.

In fact, a Lawvere theory is equivalently given by a small category $\T$ with objects $\ob\T = \NN$ in which each object $n$ carries the structure of an $n$-th power of $1$, such that these designated $n$-th powers of $1$ are standard\footnote{Many authors drop the latter condition, with the immaterial consequence that a Lawvere theory $\T$ may then carry an irrelevant specified automorphism $\pi_1:1 \rightarrow 1$ of $1$.}.
\end{ParSub}

\begin{NotnSub}
We will use each of the following notations interchangeably to denote a given Lawvere theory:
$$\T,\;\;\;\;\;\;\;(\T,\tau),\;\;\;\;\;\;\;\left(\T,T^{(-)}\right),\;\;\;\;\;\;\;(\T,T)\;.$$
We regard the last notation as a construct for naming both $\T$ and the object $T = 1$ of $\T$.
\end{NotnSub}

\begin{ParSub}[\textbf{The category of Lawvere theories}]\label{par:cat_th}
Given Lawvere theories $\T$ and $\U$, a \textbf{morphism} $M:\T \rightarrow \U$ is a functor that commutes with the associated functors $\FinCard^\op \rightarrow \T$ and $\FinCard^\op \rightarrow \U$.  Thus Lawvere theories are the objects of a category $\Th$.  Observe that $\FinCard^\op$ is a Lawvere theory, when equipped with its identity functor, and hence is an initial object of $\Th$.  A morphism of theories may be equivalently defined as a functor that strictly preserves the designated $n$-th power projections $\pi_i:n \rightarrow 1$ (and so, in particular, is identity-on-objects).  Given a Lawvere theory $\U$, a \textbf{subtheory} of $\U$ is a Lawvere theory $\T$ equipped with a morphism $\T \hookrightarrow \U$ that is faithful as a functor.  Subtheories are the objects of a full subcategory $\SubTh(\U)$ of the slice category $\Th \slash \U$, and $\SubTh(\U)$ is clearly a preordered set.  A \textbf{concrete subtheory} is a subtheory $\T \hookrightarrow \U$ for which the associated mapping $\mor\T \rightarrow \mor\U$ is simply the inclusion of a subset $\mor\T \subseteq \mor\U$.  Clearly every subtheory of $\U$ is isomorphic to a concrete subtheory.  A concrete subtheory of $(\U,U)$ is equivalently given by a subset $\T \subseteq \mor\U$ such that (i) $\T$ contains the designated projections $\pi_i:U^n \rightarrow U$, (ii) $\T$ is closed under composition, and (iii) given a family of morphisms $\{\omega_i:U^n \rightarrow U \mid i = 1,...,m\} \subseteq \T$, the induced morphism $\omega:U^n \rightarrow U^m$ lies in $\T$.
\end{ParSub}

\begin{ParSub}[$\T$-\textbf{algebras}]\label{par:talgs}
Given a Lawvere theory $(\T,T)$ and a category $\C$, a $\T$-\textbf{algebra} in $\C$ is a functor $A:\T \rightarrow \C$ that preserves finite powers.  We call $\ca{A} := A(T)$ the \textit{carrier} of $A$.  Hence for each natural number $n$, the object $A(T^n)$ is simply an $n$-th power $\ca{A}^n$ of the carrier $\ca{A}$ in $\C$.  However, when $\C$ has designated finite powers of each of its objects, the $n$-th power $A(T^n) = \ca{A}^n$ need not be the designated $n$-th power.  Assuming that $\C$ has \textit{standard} designated finite powers \pbref{par:lth}, we therefore define a \textbf{normal} $\T$-\textbf{algebra} to be a functor $A:\T \rightarrow \C$ that sends the designated $n$-th power projections $\pi_i:T^n \rightarrow T$ in $\T$ to the designated $n$-th power projections $\pi_i:\ca{A}^n \rightarrow \ca{A}$ in $\C$.  Since a functor on $\T$ preserves finite powers as soon as it preserves finite powers of $T$, every normal $\T$-algebra is necessarily a $\T$-algebra.  Observe that a morphism of Lawvere theories $A:(\T,T) \rightarrow (\U,U)$ is equivalently defined as a normal $\T$-algebra in $\U$ with carrier $A(T) = U$.  With this in mind, note also that a normal $\T$-algebra $A:\T \rightarrow \C$ is uniquely determined by its carrier and its components $A_{n,1}:\T(n,1) \rightarrow \C(\ca{A}^n,\ca{A})$, $n \in \NN$.  Related to this, observe that a morphism $A:\T \rightarrow \U$ is an isomorphism (resp. a subtheory embedding) if and only if $A_{n,1}$ is an isomorphism (resp. a monomorphism) in $\Set$ for each $n \in \NN$.
\end{ParSub}

\begin{ParSub}[\textbf{The category of} $\T$-\textbf{algebras}]\label{par:cat_talgs}
Letting $(\T,T)$ be a Lawvere theory, observe that $\T$-algebras in a given category $\C$ are the objects of a full subcategory $\Alg{\T}_\C$ of the functor category $[\T,\C]$.  We call natural transformations between $\T$-algebras $\T$-\textbf{homomorphisms}.  When $\C$ has standard designated finite powers, we denote by $\Alg{\T}_\C^!$ the full subcategory of $\Alg{\T}$ with objects all normal $\T$-algebras in $\C$.  In fact we obtain an equivalence of categories
$$\Alg{\T}_\C^! \simeq \Alg{\T}_\C$$
between normal $\T$-algebras and arbitrary $\T$-algebras; this is proved in general context in \cite[5.14]{Lu:EnrAlgTh}.

There is a canonical functor $\ca{\text{$-$}} = \Ev_T:\Alg{\T}_\C \rightarrow \C$ given by evaluation at $T$, and in fact this functor is faithful.  Indeed, given a $\T$-homomorphism $\phi:A \rightarrow B$, we know that for each $n \in \NN$, the morphisms $A(\pi_i):A(T^n) \rightarrow A(T) = \ca{A}$ present $A(T^n)$ as an $n$-th power $\ca{A}^n$ of $\ca{A} = A(T)$ in $\C$, and similarly for $B$, and the naturality of $\phi$ entails that the component $\phi_{T^n}:A(T^n) \rightarrow B(T^n)$ is simply the morphism $f^n:\ca{A}^n \rightarrow \ca{B}^n$ induced by $f := \phi_T:\ca{A} \rightarrow \ca{B}$.  Hence the mapping $\Ev_T:\Alg{\T}_\C(A,B) \rightarrow \C(\ca{A},\ca{B})$ is injective.  In fact, \textit{via this injective map we can and will identify $\Alg{\T}_\C(A,B)$ with the subset of $\C(\ca{A},\ca{B})$ consisting of all morphisms $f:\ca{A} \rightarrow \ca{B}$ such that
\begin{equation}\label{eq:f_pres_mu}
\xymatrix{
\ca{A}^n \ar[d]_{A(\mu)} \ar[r]^{f^n} & \ca{B}^n \ar[d]^{B(\mu)}\\
\ca{A}^m \ar[r]_{f^m} & \ca{B}^m
}
\end{equation}
commutes for every morphism $\mu:T^n \rightarrow T^m$ in $\T$}.  Hence we say that a morphism $f:\ca{A} \rightarrow \ca{B}$ is a $\T$-\textbf{homomorphism} if it satisfies this condition.  To assert merely that the square \eqref{eq:f_pres_mu} commutes for a particular morphism $\mu:T^n \rightarrow T^m$, we say that $f$ \textbf{preserves $\mu$ (relative to $A$ and $B$)}.  Since each such morphism $\mu$ is induced by a family of morphisms $(\mu_i:T^n \rightarrow T)_{i = 1}^m$, it follows that $f$ preserves $\mu$ iff $f$ preserves each of the $\mu_i$.  Therefore $f$ is a $\T$-homomorphism iff $f$ preserves every morphism of the form $\omega:T^n \rightarrow T$ in $\T$.
\end{ParSub}

\begin{ParSub}[\textbf{Algebraic categories}]\label{par:alg_cats}
We shall often call $\T$-algebras in $\Set$ simply $\T$-\textbf{algebras}.  We write simply $\Alg{\T}$ for the category of $\T$-algebras in $\Set$ and $\Alg{\T}^!$ for its equivalent full subcategory consisting of normal $\T$-algebras.  We say that a functor $G:\A \rightarrow \Set$ is \textbf{strictly finitary-algebraic} (or that $\A$ is strictly finitary-algebraic over $\Set$, via $G$) if there exists a Lawvere theory $\T$ and an isomorphism $(\A,G) \cong (\Alg{\T}^!,\ca{\text{$-$}})$ in the slice category $\CAT \slash \Set$.  It is well known that $G$ is strictly finitary-algebraic if and only if $G$ is strictly monadic for a finitary monad $\TT$ on $\Set$, meaning that (i) $G$ has a left adjoint $F$, (ii) the endofunctor $T = GF$ preserves filtered colimits, and (iii) the comparison functor $\A \rightarrow \Set^\TT$ for the associated monad $\TT = (T,\eta,\mu)$ is an isomorphism\footnote{Indeed, this follows from \cite[4.2, 11.3, 11.8, 11.14]{Lu:EnrAlgTh}, noting that the isomorphisms in 11.14 there are isomorphisms in $\CAT \slash \Set$ when $\V = \Set$.}.  The associated theory $\T$ is obtained by forming the Kleisli category $\Set_\TT$ and defining $\T$ to be the full subcategory of $\Set_\TT^\op$ with objects the finite cardinals, so that $\T(n,m) = \Set(m,T(n)) \cong (T(n))^m$.  Note then that $\T(n,m) \cong \A(Fm,Fn)$, with composition as in $\A$, so that we have a fully faithful functor $\T^\op \rightarrowtail \A$.  The associated isomorphism $\A \rightarrow \Alg{\T}^!$ sends each object $A$ of $\A$ to a normal $\T$-algebra $\underline{A}:\T \rightarrow \Set$ that has carrier $GA$ and associates to each abstract operation $\omega \in \T(n,1)$ the mapping $\underline{A}\omega:(GA)^n \rightarrow GA$ defined as follows.  Regarding $n$ as a cardinal, we have an associated object $Fn$ of $\A$ and an isomorphism $\A(Fn,-) \cong \Set(n,G-) = (G-)^n$.  Hence each element $a = (a_1,...,a_n) \in (GA)^n$ determines a corresponding morphism $a^\sharp:Fn \rightarrow A$ in $\A$ and an underlying mapping $Ga^\sharp:GFn \rightarrow GA$.  Recalling that $GFn = Tn = \T(n,1)$, the associated mapping $\underline{A}\omega:(GA)^n \rightarrow GA$ is given by
$$(\underline{A}\omega)(a) = (Ga^\sharp)(\omega)\;.$$
\end{ParSub}

\begin{ParSub}[\textbf{Varieties of algebras}]\label{par:var_algs}
Let us now recall some points concerning the relation of Lawvere theories to classical universal algebra.  Mac Lane \cite[V.6, p. 120]{MacL}, for example, gives a concise introduction to the basic framework of classical universal algebra, defining the category $\Alg{\langle \Omega,E\rangle}$ of $\langle \Omega,E\rangle$-algebras for what he calls simply a \textit{type} $\langle \Omega,E\rangle$, where $\Omega$ and $E$ are suitable collections of formal operations and equations, respectively.  Any category of the form $\Alg{\langle \Omega,E\rangle}$ is called a \textbf{variety of finitary algebras}.  Theorem 1 of \cite[VI.8]{MacL} shows that the forgetful functor $\ca{\text{$-$}}:\Alg{\langle \Omega,E\rangle} \rightarrow \Set$ is strictly monadic, and, as noted in \cite[IX.1, p. 209]{MacL}, $\ca{\text{$-$}}$ creates filtered colimits, so the induced monad on $\Set$ is finitary.  Therefore $\ca{\text{$-$}}:\Alg{\langle \Omega,E\rangle} \rightarrow \Set$ is strictly finitary-algebraic.  In fact, it is well-known (and now straightforward to prove) that a functor $G:\A \rightarrow \Set$ is strictly finitary-algebraic if and only if the object $(\A,G)$ of $\CAT \slash \Set$ is isomorphic to $(\Alg{\langle \Omega,E\rangle},\ca{\text{$-$}})$ for some type $\langle \Omega,E\rangle$, though we shall not make use of this fact.
\end{ParSub}

\begin{ExaSub}[\textbf{Left $R$-modules}]\label{exa:law_th_rmods}
The category $\Mod{R}$ of left $R$-modules for a ring $R$ is a variety of finitary algebras and so is isomorphic to the category of normal $\T$-algebras $\Alg{\T}^!$ for a Lawvere theory $\T$.  By \bref{par:alg_cats}, the associated theory $\T$ has $\T(n,m) = (\ca{R}^n)^m$ since $R^n$ is the free $R$-module on $n$ generators, where we write $\ca{R}$ for the underlying set of $R$.  By identifying $(\ca{R}^n)^m$ with the set $\ca{R}^{m \times n}$ of $m \times n$-matrices, we can conveniently describe composition in $\T$ as matrix multiplication.  Hence $\T$ is the category $\Mat_R$ of $R$-\textbf{matrices}, whose objects are natural numbers and whose morphisms $w:n \rightarrow m$ are $m \times n$-matrices with entries in $R$.

The normal $\T$-algebra $\Mat_R \rightarrow \Set$ corresponding to a left $R$-module $M$ necessarily sends each object $n$ to the $n$-th power $\ca{M}^n$ of the underlying set $\ca{M}$ of $M$, and we identify $\ca{M}^n$ with the set $\ca{M}^{n \times 1}$ of $n$-element column vectors with entries in $M$.  Given an $m \times n$-matrix $w:n \rightarrow m$, the associated mapping $\ca{M}^{n \times 1} \rightarrow \ca{M}^{m \times 1}$ sends a column vector $x$ to the matrix product $wx$, whose entries are the $R$-linear combinations $(wx)_j = \sum_{i = 1}^n w_{ji}x_i$ in $M$.

This all applies equally to the case where $R$ is merely a \textbf{rig} (or \textit{semiring}), i.e. a set $R$ with two monoid structures $(R,+,0)$ and $(R,\cdot,1)$ with $+$ commutative, such that $\cdot:R \times R \rightarrow R$ preserves $+$ and $0$ in each variable separately.

Recall that we chose to regard $(\ca{R}^n)^m$ as the set of $m \times n$-matrices $\ca{R}^{m \times n}$, whereas we could have considered $n \times m$-matrices instead.  This line of inquiry is pursued below in the course of our discussion on \textit{commutants} \pbref{exa:cmtnt_lth_rmods}.
\end{ExaSub}

\begin{ExaSub}[\textbf{The Lawvere theory of commutative $k$-algebras}]\label{exa:lth_ckalgs}
Given a commutative ring $k$, the category $\CAlg{k}$ of commutative $k$-algebras is a variety of finitary algebras and so is isomorphic to the category $\Alg{\T}^!$ of normal $\T$-algebras for a Lawvere theory $\T$.  Since the polynomial ring $k[x_1,...,x_n]$ is the free commutative $k$-algebra on $n$-generators, we deduce by \bref{par:alg_cats} that $\T(n,1) = k[x_1,...,x_n]$, and more generally $\T(n,m) = (\T(n,1))^m$ is the set of $m$-tuples of $n$-variable polynomials.
\end{ExaSub}

\begin{ExaSub}[\textbf{The Lawvere theory of semilattices}]\label{exa:slats}
A \textbf{(bounded) join semilattice} can be defined either as a partially ordered set with finite joins, or, equivalently, as a commutative monoid in which every element is idempotent.  Hence the category $\SLat_\vee$ of join semilattices (and their homomorphisms) is a variety of finitary algebras.   The powerset $2^n$ of a finite cardinal $n$ is the free join semilattice on $n$ generators, namely the singletons $\{i\}$ with $i \in n$.  By \bref{par:alg_cats}, $\SLat_\vee$ is therefore isomorphic to the category $\Alg{\T}^!$ of normal $\T$-algebras for a Lawvere theory $\T$ with $\T(n,m) = (2^n)^m = 2^{m \times n}$, and we find that composition in $\T$ is given by matrix multiplication when we view $2 = \{0,1\}$ as a rig \pbref{exa:law_th_rmods} with underlying additive monoid $(2,\vee,0)$ and multiplicative monoid $(2,\wedge,1)$.  Hence $\T = \Mat_2$ is the category of $2$-matrices, and $\SLat_\vee = \Mod{2}$, so that join semilattices are the same as $2$-modules by \bref{exa:law_th_rmods}.

Exchanging joins for meets, the category $\SLat_\wedge$ of \textbf{(bounded) meet semilattices} is isomorphic to $\SLat_\vee$, via an isomorphism that commutes with the forgetful functors to $\Set$.
\end{ExaSub}

\begin{DefSub}[\textbf{The full finitary theory of an object}]\label{def:full_theory}
If a given object $C$ of a locally small category $\C$ has standard designated finite powers $C^n$, $n \in \NN$, then we obtain a Lawvere theory $\C_C$, called the \textbf{full finitary theory of $C$} in $\C$, with
$$\C_C(n,m) = \C(C^n,C^m),\;\;\;\;\;\;n,m \in \ob\C_C = \NN$$
such that the mapping $\NN \rightarrow \ob\C$, $n \mapsto C^n$, extends to an identity-on-homs functor $\C_C \rightarrowtail \C$, which is evidently a $\C_C$-algebra in $\C$ with carrier $C$.

In particular, given a Lawvere theory $(\T,T)$, any $\T$-algebra $A:\T \rightarrow \C$ endows its carrier $\ca{A} = A(T)$ with standard designated finite powers $\ca{A}^n = A(T^n)$ \pbref{par:talgs}, with respect to which we can form the full finitary theory of $\ca{A}$ in $\C$, which we shall denote by $\C_A$.  The given $\T$-algebra $A$ then factors uniquely as
$$
\xymatrix{
\T \ar[dr]_A \ar@{..>}[r]^{A'} & \C_A \ar@{ >->}[d]\\
                         & \C
}
$$
where $A'$ is a morphism of Lawvere theories, given on homs just as $A$.  By abuse of notation, we write simply $A$ to denote the morphism $A'$.

In the case that $\C$ has standard designated finite powers, morphisms of Lawvere theories $\T \rightarrow \C_C$ into the full finitary theory of an object $C$ of $\C$ are evidently in bijective correspondence with normal $\T$-algebras in $\C$ with carrier $C$.
\end{DefSub}

\begin{ExaSub}[\textbf{The Lawvere theory of Boolean algebras}]\label{exa:bool}
The category $\Bool$ of Boolean algebras is a variety of finitary algebras and so is isomorphic to the category $\Alg{\T}^!$ of normal $\T$-algebras for a Lawvere theory $\T$.  In fact, it is well-known that the morphism of theories $\T \rightarrow \Set_2$ determined by the Boolean algebra $2 = \{0,1\}$ is an isomorphism between $\T$ and the full finitary theory $\Set_2$ of $2$ in $\Set$.  Indeed, this follows from \cite[III.1, Example 4]{Law:PhD}.
\end{ExaSub}

\section{Affine and convex spaces}\label{sec:aff_cvx}

\begin{ParSub}[\textbf{The affine core of a Lawvere theory}]
Every Lawvere theory $(\T,T)$ has a subtheory $\T^{\aff} \hookrightarrow \T$ called the \textbf{affine core} of $\T$ \cite[\S 3]{Law:ProbsAlgTh}, consisting of all those morphisms $\omega:T^n \rightarrow T^m$ for which the composite
$$T \xrightarrow{(1_T,...,1_T)} T^n \overset{\omega}{\longrightarrow} T^m$$
equals $(1_T,...,1_T):T \rightarrow T^m$.  We say that $\T$ is \textbf{affine} if $\T$ equals its affine core.
\end{ParSub}

\begin{ParSub}[\textbf{Affine spaces over a ring or rig}]\label{exa:raff}
Let $R$ be a ring, or more generally, a rig.  Recall that the category of $R$-matrices $\Mat_R$ is the Lawvere theory of left $R$-modules \pbref{exa:law_th_rmods}.  By definition, a \textbf{(left)} $R$-\textbf{affine space} (or \textbf{(left)} $R$-\textbf{affine module}) is a normal $\T$-algebra for the affine core $\T = \Mat_R^\aff$ of $\Mat_R$.  Hence $R$-affine spaces are the objects of a category $\Aff{R} = \Alg{\Mat_R^\aff\:}^!$.  Since the projection morphisms $\pi_i:n \rightarrow 1$ in $\Mat_R$ are the standard basis vectors for $R^{1 \times n}$, one deduces that the affine part $\Mat_R^\aff$ consists of the $R$-matrices in which each row sums to $1$.  By \bref{par:talgs}, an $R$-affine space $E$ is therefore given by a set $E$ (the carrier) equipped with a suitable family of mappings $\Mat_R^\aff(n,1) \rightarrow \Set(E^n,E)$ that associate to each $n$-element row vector $w = [w_1,...,w_n]$ with $\sum_{i = 1}^n w_i = 1$ a mapping $E^n \rightarrow E$ whose value at a given column vector $x \in E^n$ we write as $\sum_{i = 1}^n w_ix_i$ and call a \textbf{(left)} $R$-\textbf{affine combination} of the $x_i$.  For example, since the morphism of theories $\Mat_R^\aff \hookrightarrow \Mat_R$ induces a functor $\Mod{R} \rightarrow \Aff{R}$, every left $R$-module $M$ carries the structure of a left $R$-affine space.  The morphisms in the category of (left) $R$-affine spaces $\Aff{R}$ are \textbf{(left)} $R$-\textbf{affine maps}, i.e. those mappings that preserve left $R$-affine combinations.
\end{ParSub}

\begin{ExaSub}[\textbf{Unbounded semilattices as affine spaces}]\label{exa:unb_jslats}
An \textbf{unbounded join semilattice} may be defined as a poset in which every pair of elements has a join or, equivalently, as a commutative semigroup in which every element is idempotent.  The set $2^n \backslash \{0\}$ of all nonempty subsets of a given finite cardinal $n$ is closed under binary joins in the semilattice $2^n$ \pbref{exa:slats} and so carries the structure of an unbounded join semilattice, and the singletons $\{i\}$ with $i \in n$ exhibit $2^n \backslash \{0\}$ as the free unbounded join semilattice on $n$ generators.  The category $\USLat_\vee$ of unbounded semilattices and their homomorphisms is a variety of finitary algebras and so is isomorphic to the category of normal $\T$-algebras for a Lawvere theory $\T$ with $\T(n,m) = (2^n \backslash \{0\})^m$.  The latter set may be identified with the subset of $2^{m \times n}$ consisting of all $m \times n$-matrices in which each row is nonzero, whereupon we deduce that $\T$ is a subtheory of the theory $\Mat_2$ of modules over the rig $(2,\vee,0,\wedge,1)$, i.e. (bounded) semilattices \pbref{exa:slats}.  Indeed, $\T$ is precisely the theory $\Mat_2^\aff$ of affine spaces over the rig $(2,\vee,0,\wedge,1)$.  Hence $\USLat_\vee \cong \Aff{2}$, i.e. unbounded semilattices are the same as affine spaces over the rig $(2,\vee,0,\wedge,1)$.
\end{ExaSub}

\begin{ParSub}[\textbf{Preordered abelian groups and preordered rings}]\label{par:pr_rings}
By definition, a \textbf{preordered commutative monoid} is a commutative monoid object in the cartesian monoidal category $\Ord$ of preordered sets and monotone maps.  Preordered commutative monoids are the objects of a category $\CMon(\Ord)$, the category of commutative monoids in $\Ord$.  We say that a preordered commutative monoid $M$ is a \textbf{preordered abelian group} if its underlying commutative monoid (in $\Set$) is an abelian group.  Note then that the negation map $-:M \rightarrow M$ is not monotone but rather is order-reversing.  Preordered abelian groups form a full subcategory $\Ab_\lt$ of the category of preordered commutative monoids.

It is well-known that the notion of preordered abelian group can be equivalently defined as an abelian group $M$ equipped with a submonoid $M_+ \hookrightarrow M$.  Indeed, given a preordered abelian group $M$, one takes $M_+ = \{m \in M \mid 0 \lt m\}$, and conversely, given a submonoid $M_+$ of an abelian group $M$ one defines a preorder $\lt$ on $M$ by $m \lt m'\;\Leftrightarrow m' - m \in M_+$.  It is conventional to call $M_+$ the \textbf{positive part} of $M$ despite the fact that $0 \in M_+$.  Morphisms of preordered abelian groups can be described equivalently as homomorphisms of the underlying abelian groups $h:M \rightarrow N$ with the property that $h(M_+) \subseteq N_+$.

The category $\Ab_\lt$ of preordered abelian groups is symmetric monoidal when we define the monoidal product $M \otimes N$ of preordered abelian groups $M$ and $N$ to be the usual tensor product of abelian groups equipped with the submonoid $(M \otimes N)_+ \hookrightarrow M \otimes N$ generated by the pure symbols $m \otimes n$ with $m \in M_+$, $n \in N_+$.  The unit object is $\ZZ$, with the natural order.

By definition, a \textbf{preordered ring} is a monoid in the monoidal category of preordered abelian groups $\Ab_\lt$.  Equivalently, a preordered ring is a ring $R$ equipped with an arbitrary subrig $R_+ \hookrightarrow R$.  Note that any rig $S$ that occurs as the positive part $R_+$ of some preordered ring $R$ is necessarily \textbf{additively cancellative}, meaning that the commutative semigroup $(S,+)$ is cancellative (i.e., $s + t = s + u\;\Rightarrow\;t = u$).  Moreover, the rigs that occur as positive parts of preordered rings are precisely the additively cancellative rigs, since if $S$ is additively cancellative then we can embed $S$ into its \textit{ring completion}, mimicking the usual construction of $\ZZ$ from $\NN$.

Preordered rings are the objects of a category $\Ring_\lt$, the category of monoids in the monoidal category $\Ab_\lt$, in which the morphisms are ring homomorphism that are also monotone.
\end{ParSub}

\begin{DefSub}[\textbf{Convex spaces over a preordered ring}]
Given a preordered ring $R$, a \textbf{(left)} $R$-\textbf{convex space} (or \textbf{(left)} $R$-\textbf{convex module}) is a left $R_+$-affine space, i.e. a left affine space over the rig $R_+$ obtained as the positive part of $R$ \pbref{par:pr_rings}.  Hence an $R$-convex space is by definition a set equipped with operations that permit the taking of left $R_+$-affine combinations \pbref{exa:raff}, which we call \textbf{(left)} $R$-\textbf{convex combinations}.  We write $\Cvx{R} := \Aff{R_+}$ for the category of $R$-convex spaces.  Note that $R$-convex spaces are the normal $\T$-algebras for the Lawvere theory $\T = \Mat_{R_+}^\aff$, whose morphisms are $R_+$-matrices in which each row sums to $1$.
\end{DefSub}

\begin{ExaSub}[\textbf{Convex spaces over the reals}]
For example, when $R$ is the real numbers $\RR$ with the usual order, the notion of $\RR$-convex space is the familiar notion of \textit{convex space}.  Observe that when $n > 0 $, $\Mat_{\RR_+}^\aff(n,1) \subseteq \RR^{1 \times n}$ is the standard geometric $(n-1)$-simplex, presented in terms of barycentric coordinates.
\end{ExaSub}

\begin{ExaSub}[\textbf{Convex spaces over a ring of continuous functions}]
Given a topological space $X$, let $C(X)$ denote the ring of all real-valued continuous functions on $X$.  The pointwise partial order on $C(X)$ makes it a preordered ring whose positive part $C(X)_+$ is the set $C(X,\RR_+)$ of all continuous $\RR_+$-valued functions.  Given any convex subset $S$ of $\RR^n$, the set $C(X,S)$ of all continuous $S$-valued functions on $X$ carries the structure of a $C(X)$-convex space.
\end{ExaSub}

\section{Commutation and Kronecker products of operations}\label{sec:cmtn}

\begin{ParSub}\label{par:binary_prods_in_fincard}
Given a pair of finite cardinals $(j,k)$, regarded as sets, the cartesian product $j \times k$ in $\Set$ has cardinality $jk$, so $jk$ serves as a product in $\FinCard$ of the objects $j$ and $k$.  Despite this apparently simple way of forming binary products in $\FinCard$, one must also choose projections $\pi^{(j,k)}_1:jk \rightarrow j$ and $\pi^{(j,k)}_2:jk \rightarrow k$ if $jk$ is to be equipped with the structure of a product of $(j,k)$ in $\FinCard$, and there is more than one way to do this.  In the present section, we must fix a determinate choice of such product projections---equivalently, we must fix designated bijections $j \times k \rightarrow jk$ in $\Set$ that let us \textit{encode} elements of the cartesian product $j \times k$ as elements of $jk$.  For example, we shall see in \bref{exa:kp_matr} that the definition of the classical \textit{Kronecker product} of matrices depends on one specific such encoding \eqref{eq:kp_prod_enc}.  With such a choice, $\FinCard$ is cartesian monoidal.  When $jk$ is to be regarded as a product of $(j,k)$ via the chosen projections we shall denote it by $j \times k$.  It is important to note that the symmetry isomorphism $s_{j,k}:j \times k \rightarrow k \times j$ is \textit{not} in general the identity map on $jk = kj$.  We shall take the designated projections $\pi_1:j = j \times 1 \rightarrow j$ and $\pi_2:k = 1 \times k \rightarrow k$ to be the identity maps.
\end{ParSub}

\begin{ParSub}\label{par:bif_kp}
Let $(\T,T)$ be a Lawvere theory.  Since the product $jk$ of natural numbers $j$ and $k$ carries the structure of a product $j \times k$ in $\FinCard$ \pbref{par:binary_prods_in_fincard}, there is an associated canonical way of equipping the object $T^{j \times k}$ of $\T$ with the structure of both a $j$-th power of $T^k$ and also a $k$-th power of $T^j$, which we may signify informally by writing
\begin{equation}\label{eq:lr_powers}(T^k)^j = T^{j \times k} = (T^j)^k\;.\end{equation}
Writing $p^{(j,k)} = (p^{(j,k)}_v:T^{j \times k} \rightarrow T^k)_{v = 1}^j$ and $q^{(j,k)} = (q^{(j,k)}_t:T^{j \times k} \rightarrow T^j)_{t = 1}^k$ for the associated power cones, we call $(T^{j \times k},p^{(j,k)})$ the \textbf{left} $j$-\textbf{th power} of $T^k$, and we call $(T^{j \times k},q^{(j,k)})$ the \textbf{right} $k$-\textbf{th power}\footnote{It is instructive to note that the right $k$-th power cone $q^{(j,k)}$ is not in general equal to the left $k$-th power cone $p^{(k,j)}$, despite the fact that these cones both equip the same object $T^{jk}$ with the structure of a $k$-th power of $T^j$.  Indeed, the automorphism of $T^{jk}$ induced by this pair of $k$-th power cones is the isomorphism $T^{k \times j} \rightarrow T^{j \times k}$ determined by the symmetry isomorphism $j \times k \rightarrow k \times j$ in $\FinCard$ \pbref{par:binary_prods_in_fincard}.} of $T^j$.  Writing $T^j * T^k = T^{j \times k}$, we therefore obtain evident functors 
$$T^j * (-):\T \rightarrow \T,\;\;\;\;T^k \mapsto T^{j \times k}\;\;\;\;\;\;\;\;\;(j \in \NN)$$
$$(-) * T^k:\T \rightarrow \T,\;\;\;\;T^j \mapsto T^{j \times k}\;\;\;\;\;\;\;\;\;(k \in \NN)$$
induced by left $j$-th and right $k$-th power structures, respectively, carried by the objects $T^{j \times k}$ with $(j,k) \in \NN \times \NN$.  This now begs the question as to whether these are the partial functors of a bifunctor
$$*\;\;:\;\;\T \times \T \rightarrow \T$$
given on objects by $(T^j,T^k) \mapsto T^{j \times k}$.  By \cite[II.3, Prop. 1]{MacL}, this is the case if and only if for every pair of morphisms $\mu:T^j \rightarrow T^{j'}$ and $\nu:T^k \rightarrow T^{k'}$ in $\T$ the composites
\begin{equation}\label{eq:kps}
\begin{array}{ll}
1. & T^j * T^k \xrightarrow{\mu * T^k} T^{j'} * T^k \xrightarrow{T^{j'} * \nu} T^{j'} * T^{k'}\\
2. & T^j * T^k \xrightarrow{T^j * \nu} T^j * T^{k'} \xrightarrow{\mu * T^{k'}} T^{j'} * T^{k'}
\end{array}
\end{equation}
are equal.  This leads to the following:
\end{ParSub}

\begin{DefSub}\label{def:kps_cmt}\emptybox
\begin{enumerate}
\item Given morphisms $\mu:T^j \rightarrow T^{j'}$ and $\nu:T^{k} \rightarrow T^{k'}$ in $\T$, the \textbf{first and second Kronecker products} $\mu * \nu$ and $\mu \stt \nu$ of $\mu$ with $\nu$ are defined as the composites 1 and 2 in \eqref{eq:kps}, respectively, i.e.
$$\mu * \nu = \left(T^{j \times k} \xrightarrow{\mu * T^k} T^{j' \times k} \xrightarrow{T^{j'} * \nu} T^{j' \times k'}\right)\;,$$
$$\mu \stt \nu = \left(T^{j \times k} \xrightarrow{T^j * \nu} T^{j \times k'} \xrightarrow{\mu * T^{k'}} T^{j' \times k'}\right)\;.$$
\item We say that $\mu$ \textbf{commutes with} $\nu$ if $\mu * \nu = \mu \stt \nu$.
\item We say that $\T$ is \textbf{commutative} if $\mu$ commutes with $\nu$ for every pair of morphisms $\mu$ and $\nu$ in $\T$.
\end{enumerate}
\end{DefSub}

\begin{ExaSub}[\textbf{The Kronecker product of matrices}]\label{exa:kp_matr}
Given a ring $R$, or even just a rig $R$, consider the Lawvere theory of left $R$-modules, i.e. the category $\Mat_R$ of $R$-matrices \pbref{exa:law_th_rmods}.  Recall that morphisms $j \rightarrow j'$ in $\Mat_R$ are $j' \times j$-matrices.  Since such matrices are usually indexed by pairs of \textit{positive} integers, it shall be convenient here to depart from the usual von Neumann definition of the ordinals and instead identify each object $j$ of $\FinCard$ with the set of all positive integers less than or equal to $j$, so that the above $j' \times j$-matrices are families indexed by the usual cartesian product\footnote{putting aside for the moment our similar notation for the product in $\FinCard$ \pbref{par:binary_prods_in_fincard}.} $j' \times j$.  

Letting $X \in \Mat_R(j,j') = R^{j'\times j}$ and $Y \in \Mat_R(k,k') = R^{k' \times k}$, the classical \textbf{Kronecker product} of $Y$ by $X$ is the $j'k' \times jk$-matrix $Y \otimes X$ with entries
\begin{equation}\label{eq:kp_matr}(Y \otimes X)_{\langle u,s\rangle\langle v,t\rangle} = Y_{st}X_{uv}\;\;\;\;\;\;\;\;u \in j',\:s \in k',\:v \in j,\:t \in k\;.\end{equation}
where in general we write $\langle v, t \rangle$ to denote the element 
\begin{equation}\label{eq:kp_prod_enc}\langle v, t \rangle = v + j(t - 1)\end{equation}
of $jk$ associated to the pair $(v,t) \in j \times k$.  The hidden reason behind the seemingly arbitrary convention \eqref{eq:kp_prod_enc} is that it provides one standard way of assigning to each pair of finite cardinals $(j,k)$ a bijection $\langle -,-\rangle:j \times k \rightarrow jk$, so that $jk$ is thus equipped with the structure of a product of $(j,k)$ in $\FinCard$.  Indeed, each element of $jk$ can be written in the form $\langle v, t \rangle$ for unique $v \in j$ and $t \in k$, and the maps $\pi_1:jk \rightarrow j$ and $\pi_2:jk \rightarrow k$ given by $\pi_1(\langle v, t\rangle) = v$ and $\pi_2(\langle v, t\rangle) = t$ present $jk$ as a product $j \times k$ in $\Set$ and hence in $\FinCard$.

With this choice of binary products in $\FinCard$, we claim that
$$X * Y = Y \otimes X\;,$$
i.e., the first Kronecker product $X * Y$ in $\Mat_R$ is the usual Kronecker product of matrices $Y \otimes X$.

In order to prove this, recall that $X * Y$ is defined as the composite
$$jk \xrightarrow{X * k} j'k \xrightarrow{j' * Y} j'k'$$
in $\Mat_R$, where here we write the objects $T^n$ of the theory $\T = \Mat_R$ simply as $n$ (since concretely $T^n = n$).  In order to examine the entries of the matrices $X * k$ and $j' * Y$, let us first note that for each object $n$ of $\Mat_R$, the designated $n$-th power cone $(P_i:n \rightarrow 1)_{i \in n}$ in $\Mat_R$ consists of the standard basis row-vectors $P_i \in R^{1 \times n}$, having a $1$ in the $i$-th position and zeros everywhere else.  In particular, $jk$ is a $jk$-th power of $1$ in $\Mat_R$, via the morphisms $P_{\langle v,t\rangle}:jk \rightarrow 1$ with $(v,t) \in j \times k$.  For fixed $t \in k$, these morphisms induce a unique morphism $Q^{(j,k)}_t:jk \rightarrow j$ in $\Mat_R$ such that $P_vQ^{(j,k)}_t = P_{\langle v,t\rangle}$ $(v \in j)$, and the resulting family\footnote{This family is the \textit{right $k$-th power cone} $q^{(j,k)}$ in the terminology of \bref{par:bif_kp}.} $(Q^{(j,k)}_t)_{t \in k}$ presents $jk$ as a $k$-th power of $j$ in $\Mat_R$.  Explicitly, $Q^{(j,k)}_t$ is the $j \times jk$-matrix with entries
$$(Q^{(j,k)}_t)_{vb} = \{\;1\;\text{if $b = \langle v,t\rangle$}, \;\;0\;\text{otherwise}$$
where $v \in j$ and $b \in jk$.  Similarly, morphisms $(Q^{(j',k)}_t:j'k \rightarrow j')_{t \in k}$ present $j'k$ as a $k$-th power of $j'$ in $\Mat_R$.  By definition, $X * k:jk \rightarrow j'k$ is the unique morphism such that $jk \xrightarrow{X * k} j'k \xrightarrow{Q^{(j',k)}_t} j'$ equals $jk \xrightarrow{Q^{(j,k)}_t} j \xrightarrow{X} j'$ for all $t \in k$.  It is straightforward to show therefore that $X * k$ is the $j'k \times jk$-matrix with entries
$$(X * k)_{\langle u,t_1\rangle\langle v,t_2\rangle} = \{\;X_{uv}\;\text{if $t_1 = t_2$}, \;\;0\;\text{otherwise},$$
where $u \in j'$, $v \in j$, and $t_1,t_2 \in k$.  Analogously\footnote{This time we use the \textit{left $j'$-th power cones} $p^{(j',k)}$ and $p^{(j',k')}$ in the terminology of \bref{par:bif_kp}.}, $j' * Y:j'k \rightarrow j'k'$ is the $j'k' \times j'k$-matrix with entries
$$(j' * Y)_{\langle u_1,s\rangle\langle u_2,t\rangle} = \{\;Y_{st}\;\text{if $u_1 = u_2$}, \;\;0\;\text{otherwise},$$
where $u_1,u_2 \in j'$, $s \in k'$ and $t \in k$.  But $X * Y$ is the matrix product $(j' * Y)(X * k) \in R^{j'k' \times jk}$, and one now readily computes that the entries of the latter product are exactly those of the Kronecker product $Y \otimes X$ (cf. \bref{eq:kp_matr}).
\end{ExaSub}

\begin{PropSub}[\textbf{Relation between the first and second Kronecker products}]\label{prop:rel_kps}
Given morphisms $\mu:T^j \rightarrow T^{j'}$ and $\nu:T^k \rightarrow T^{k'}$ in a Lawvere theory $\T$, the second Kronecker product $\mu \stt \nu$ can be expressed in terms of the first Kronecker product $\nu * \mu$ via the commutativity of the diagram
$$
\xymatrix{
T^{j \times k} \ar[d]_\wr \ar[r]^{\mu \stt \nu} & T^{j' \times k'} \ar[d]^\wr\\
T^{k \times j} \ar[r]_{\nu * \mu} & T^{k' \times j'}
}
$$
in which the left and right sides are the isomorphisms induced by the symmetry isomorphisms $k \times j \rightarrow j \times k$ and $k' \times j' \rightarrow j' \times k'$ in $\FinCard$ \pbref{par:binary_prods_in_fincard}.  As a consequence, the commutation relation is symmetric, i.e.
$$
\begin{minipage}{4in}
\textit{$\mu$ commutes with $\nu$ if and only if $\nu$ commutes with $\mu$.}
\end{minipage}
$$
\end{PropSub}
\begin{proof}
This follows from the fact that for each fixed $n \in \NN$, the symmetry isomorphisms $m \times n \rightarrow n \times m$ in $\FinCard$ \pbref{par:binary_prods_in_fincard} with $m \in \NN$ induce isomorphisms $T^{n \times m} \rightarrow T^{m \times n}$ in $\T$ that constitute a \textit{natural} isomorphism $T^n * (-) \Rightarrow (-) * T^n$.
\end{proof}

\begin{ExaSub}[\textbf{The second Kronecker product of matrices}]\label{exa:second_kp_matr}
Continuing Example \bref{exa:kp_matr}, it now follows from \bref{prop:rel_kps} that the second Kronecker product $X \stt Y:jk \rightarrow j'k'$ of morphisms $X:j \rightarrow j'$ and $Y:k \rightarrow k'$ in $\Mat_R$ is the $j'k' \times jk$-matrix with entries
$$(X \stt Y)_{\langle u,s \rangle\langle v,t\rangle} = (Y * X)_{\langle s,u\rangle\langle t,v\rangle} = (X \otimes Y)_{\langle s,u\rangle\langle t,v\rangle} = X_{uv}Y_{st}$$
where $u \in j'$, $s \in k'$, $v \in j$, $t \in k$.  Hence if $R$ is commutative then $X \stt Y = Y \otimes X = X * Y$, showing that $\Mat_R$ is commutative.  Conversely, if $\Mat_R$ is commutative then by taking $j = j' = k = k' = 1$ we find that $R$ is commutative.  Hence we have proved the following:
\begin{equation}\label{eq:lth_rmods_cmtv_iff_r_cmtv}
\begin{minipage}{3.9in}
\textit{The Lawvere theory $\Mat_R$ of left $R$-modules for a rig $R$ is commutative if and only if $R$ is commutative.}
\end{minipage}
\end{equation}
In particular, the Lawvere theory $\Mat_2$ of semilattices is commutative.
\end{ExaSub}

Clearly any subtheory of a commutative theory is commutative.  In particular, the following Lawvere theories are commutative, as each is a subtheory of a theory of the form $\Mat_R$ for a commutative rig $R$:

\begin{ExaSub}\label{exa:commutative_lths}
The following Lawvere theories are commutative:
\begin{enumerate}
\item The theory $\Mat_R^\aff$ of $R$-affine spaces for a commutative ring or rig $R$.
\item The theory of $R$-convex spaces $\Mat_{R_+}^\aff$ for a commutative preordered ring $R$.
\item The theory of unbounded semilattices $\Mat_2^\aff$.
\end{enumerate}
\end{ExaSub}

\begin{ParSub}
Let $(\T,T)$ be a Lawvere theory.  Given $j \in \NN$, any choice of $j$-th powers in $\T$ determines an endofunctor $(-)^j:\T \rightarrow \T$, and since this endofunctor $(-)^j$ preserves finite powers it can be regarded as a $\T$-algebra in $\T$.  In particular, if we employ the \textit{left $j$-th powers} in $\T$ \pbref{par:bif_kp}, then the resulting endofunctor is the functor $T^j * (-):\T \rightarrow \T$ of \bref{par:bif_kp}, given on objects by $T^k \mapsto T^{j \times k}$.  Therefore $T^j * (-)$ is a $\T$-algebra in $\T$ with carrier $T^{j \times 1} = T^j$.  We employ this observation in the following:
\end{ParSub}

\begin{PropSub}\label{thm:cmtn_via_presn_and_cpnts}
Let $\mu:T^j \rightarrow T^{j'}$ and $\nu:T^k \rightarrow T^{k'}$ be morphisms in a Lawvere theory $(\T,T)$, and denote by $(\mu_u:T^j \rightarrow T)_{u = 1}^{j'}$ and $(\nu_s:T^k \rightarrow T)_{s = 1}^{k'}$ the families inducing $\mu$ and $\nu$, respectively.  Then we have $\T$-algebras $A = T^j * (-)$ and $B = T^{j'} * (-)$ in $\T$ with carriers $\ca{A} = T^j$, $\ca{B} = T^{j'}$, respectively, and the following are equivalent:
\begin{enumerate}
\item $\mu$ commutes with $\nu$.
\item $\mu:\ca{A} \rightarrow \ca{B}$ preserves $\nu$ relative to $A$ and $B$ \pbref{par:cat_talgs}.
\item $\mu_u:T^j \rightarrow T$ commutes with $\nu_s:T^k \rightarrow T$ for all indices $u$ and $s$.
\item $\mu$ commutes with each of the components $\nu_s$ of $\nu$.
\end{enumerate}
\end{PropSub}
\begin{proof}
The equivalence 1 $\Leftrightarrow$ 2 follows readily from the definitions.  By \bref{par:cat_talgs}, $2$ is equivalent to the following statement:
\begin{enumerate}
\item[5.] $\mu:\ca{A} \rightarrow \ca{B}$ preserves each of the components $\nu_s$ of $\nu$.
\end{enumerate}
Now invoking the equivalence 1 $\Leftrightarrow$ 2 with respect to the morphisms $\mu$ and $\nu_s$, we deduce that 5 is equivalent to 4.  By symmetry, 4 holds iff each component $\nu_s$ commutes with $\mu$.  Having established the equivalence of 1 and 4 for an arbitrary pair of morphisms $(\mu,\nu)$, we can now invoke this equivalence with respect to each pair $(\nu_s,\mu)$ to deduce that 4 holds if and only if each $\nu_s$ commutes with each of the components $\mu_u$ of $\mu$, and by symmetry this is equivalent to 3.
\end{proof}

\begin{RemSub}
Having reduced the notion of commutation of morphisms $\mu$ and $\nu$ in a Lawvere theory $(\T,T)$ to the case of morphisms of the form $\mu:T^j \rightarrow T$ and $\nu:T^k \rightarrow T$, observe that the definition of the first and second Kronecker products of such morphisms reduces to the following:
$$\mu * \nu = \left(T^{j \times k} \xrightarrow{\mu * T^k} T^k \xrightarrow{\nu} T\right),\;\;\;\;\;\;\;\;\mu \stt \nu = \left(T^{j \times k} \xrightarrow{T^j * \nu} T^j \xrightarrow{\mu} T\right).$$
\end{RemSub}

\section{Commutants}\label{sec:cmtnt}

\begin{DefSub}\label{def:cmtn_mor_th}
Let $\U$ be a Lawvere theory.
\begin{enumerate}
\item Letting $A: \T \rightarrow \U$ and $B:\sS \rightarrow \U$ be morphisms of Lawvere theories, we say that $A$ \textbf{commutes with} $B$ (or that $A$ and $B$ commute) if $A(\mu)$ commutes with $B(\nu)$ in $\U$ for all morphisms $\mu$ in $\T$ and $\nu$ in $\sS$.
\item A \textbf{Lawvere theory over $\U$} is a Lawvere theory $\T$ equipped with a morphism $\T \rightarrow \U$.  The \textbf{category of Lawvere theories over $\U$} is the slice category $\Th \slash \U$.
\item Given Lawvere theories $\T$ and $\sS$ over $\U$, we say that $\T$ \textbf{commutes with} $\sS$ if the associated morphisms to $\U$ commute.
\item Subtheories $\T$ and $\sS$ of $\U$ are said to \textbf{commute} if they commute as Lawvere theories over $\U$, i.e. if $\mu$ commutes with $\nu$ for all $\mu \in \mor\T$ and $\nu \in \mor\sS$.
\end{enumerate}
\end{DefSub}

\begin{PropSub}
Let $A:(\T,T) \rightarrow (\U,U)$ and $B:(\sS,S) \rightarrow (\U,U)$ be morphisms of Lawvere theories.  Then $A$ commutes with $B$ if and only if $A(\mu)$ commutes with $B(\nu)$ in $\U$ for all morphisms of the form $\mu:T^j \rightarrow T$ in $\T$ and $\nu:S^k \rightarrow S$ in $\sS$.
\end{PropSub}
\begin{proof}
This follows from \bref{thm:cmtn_via_presn_and_cpnts}, since $A$ and $B$ strictly preserve the designated finite powers of $T$ and $S$, respectively.
\end{proof}

\begin{DefSub}
A morphism of Lawvere theories $A:\T \rightarrow \U$ is said to be \textbf{central} if it commutes with the identity morphism $1_\U:\U \rightarrow \U$.
\end{DefSub}

\begin{ExaSub}\label{exa:uniq_mor_from_initial_th_is_central}
Given a Lawvere theory $(\T,T)$, the unique morphism of Lawvere theories $T^{(-)}:\FinCard^\op \rightarrow \T$ is central.  Indeed, a morphism of the form $T^j \rightarrow T$ in $\T$ lies in the image of $T^{(-)}$ if and only if it is a projection $\pi_i:T^j \rightarrow T$, and given any operation $\nu:T^k \rightarrow T$ in $\T$, the diagram
$$
\xymatrix{
T^{j \times k} \ar[d]_{\pi_i * T^k} \ar[r]^{T^j * \nu} & *!<-3.5ex,0ex>{\;T^j = T^{j \times 1}} \ar[d]^{\pi_i}\\
T^k \ar[r]^\nu & T
}
$$
commutes since its left and right sides are equally the left $j$-th power projections $p^{(j,k)}_i$ and $p^{(j,1)}_i$ \pbref{par:bif_kp}, respectively.
\end{ExaSub}

\begin{PropSub}\label{thm:concrete_cmtnt}
Given a set of morphisms $\Omega \subseteq \mor\U$ in a Lawvere theory $(\U,U)$, the set of morphisms
$$\Omega^\perp = \{\mu \in \mor\U \mid \textnormal{$\mu$ commutes with every $\nu \in \Omega$}\}$$
is a concrete subtheory of $\U$ \pbref{par:cat_th}.
\end{PropSub}
\begin{proof}
Using the functoriality of $(-) * U^k:\U \rightarrow \U$ for each $k \in \NN$, one computes straightforwardly that $\Omega^\perp$ is closed under composition in $\U$.  By \bref{exa:uniq_mor_from_initial_th_is_central}, $\Omega^\perp$ contains all the projections $\pi_i:U^j \rightarrow U$.  Lastly, given a family of morphisms $\{\mu_i:U^j \rightarrow U \mid i = 1,...,j'\} \subseteq \Omega^\perp$, the induced morphism $\mu:U^j \rightarrow U^{j'}$ lies in $\Omega^\perp$ by \bref{thm:cmtn_via_presn_and_cpnts}.
\end{proof}

\begin{DefSub}\label{def:cmtnt}
Let $\U$ be a Lawvere theory.
\begin{enumerate}
\item Given a set of morphisms $\Omega \subseteq \mor\U$, we call the subtheory $\Omega^\perp \hookrightarrow \U$ of \bref{thm:concrete_cmtnt} the \textbf{commutant} of $\Omega$ (in $\U$).
\item Given a morphism of Lawvere theories $A:\T \rightarrow \U$, the \textbf{commutant} $\T^\perp_A$ of $A$ (or of $\T$ with respect to $A$) is defined as the commutant of the image $A(\mor\T) \subseteq \mor\U$.
\item Given a Lawvere theory $\T$ over $\U$, the \textbf{commutant} $\T^\perp$ of $\T$ is defined as the commutant of the associated morphism $\T \rightarrow \U$.
\item The \textbf{commutant} of a subtheory $\T \hookrightarrow \U$ is defined as the commutant $\T^\perp$ of $\T$, considered as a theory over $\U$.  Equivalently, $\T^\perp$ is the commutant of $\mor\T \subseteq \mor\U$.
\end{enumerate}
\end{DefSub}

The following is immediate from the definitions:

\begin{PropSub}\label{thm:univ_prop_cmtnt}
Let $A:\T \rightarrow \U$ and $B:\sS \rightarrow \U$ be morphisms of Lawvere theories.  Then $A$ and $B$ commute if and only if $B$ factors through the commutant $\T^\perp_A \hookrightarrow \U$ of $A$.
\end{PropSub}

\begin{RemSub}
By \bref{thm:univ_prop_cmtnt}, the commutant $\T^\perp$ of a theory $\T$ over $\U$ is characterized, up to isomorphism, by a universal property.  Hence we will sometimes also call any theory over $\U$ isomorphic to $\T^\perp$ \textbf{the commutant} of $\T$.
\end{RemSub}

\begin{ThmSub}\label{thm:cmtnt_full_th_alg}
Given a morphism of Lawvere theories $A:(\T,T) \rightarrow (\U,U)$, let us regard $A$ as a $\T$-algebra in $\U$.  Then the commutant $\T^\perp_A$ of $A$ is isomorphic to the full finitary theory of $A$ in the category $\Alg{\T}_\U$ of $\T$-algebras in $\U$.  In symbols,
$$\T^\perp_A \cong (\Alg{\T}_\U)_A$$
as theories over $\U$.  Further, we can choose standard designated finite powers in $\Alg{\T}_\U$ in such a way that this isomorphism is an identity.
\end{ThmSub}
\begin{proof}
Given an arbitrary object $U^k$ of $\U$, note that the left $j$-th powers $(U^k)^j = U^{j \times k}$ of $U^k$ $(j \in \NN)$ are standard \pbref{par:lth}, and for $k = 1$ these are precisely the designated $j$-th powers $U^j$ of $U$ in $\U$.  Let us now use these finite powers as our designated finite powers in $\U$ \pbref{par:lth}, calling them the \textit{left finite powers}.  The resultant pointwise finite powers in $\Alg{\T}_\U$ are standard, and we shall use them in forming the full finitary theory $(\Alg{\T}_\U)_A$ of $A$ in $\Alg{\T}_\U$.  Observe that for each $j \in \NN$, the designated $j$-th power $A^j$ of $A$ in $\Alg{\T}_\U$ is therefore the composite
$$\T \xrightarrow{A} \U \xrightarrow{U^j * (-)} \U$$
whose second factor is the endofunctor of $\U$ induced by the left $j$-th powers in $\U$ \pbref{par:bif_kp}.  In particular, $A^j$ has carrier $U^j$ and is given on objects by $A^j(T^k) = U^{j \times k}$.

The faithful functor $\ca{\text{$-$}}:\Alg{\T}_\U \rightarrow \U$ strictly preserves the designated finite powers and so induces a subtheory embedding $(\Alg{\T}_\U)_A \hookrightarrow \U_{|A|} = \U_{U} = \U$, and we shall now show that this is precisely the subtheory inclusion $\T^\perp \hookrightarrow \U$.  Fix a pair of natural numbers $j,j'$.  Per \bref{par:cat_talgs}, we have identified $\T$-homorphisms $A^j \rightarrow A^{j'}$ with certain morphisms $\mu:\ca{A^j} \rightarrow \ca{A^{j'}}$ in $\U$ between the carriers of the $\T$-algebras $A^j$ and $A^{j'}$, namely those $\mu$ that preserve each operation $\nu:T^k \rightarrow T^{k'}$ in $\T$.  But by using preceding description of the designated powers $A^j$ and $A^{j'}$, we find that $\mu$ preserves $\nu$ iff the following diagram commutes
$$
\xymatrix{
U^{j \times k} \ar[d]_{U^j \;*\; A(\nu)} \ar[r]^{\mu * U^k} & U^{j' \times k} \ar[d]^{U^{j'} * \;A(\nu)}\\
U^{j \times k'} \ar[r]_{\mu * U^{k'}} & U^{j' \times k'}
}
$$
i.e. iff $\mu:U^j \rightarrow U^{j'}$ commutes with $A(\nu):U^k \rightarrow U^{k'}$.
\end{proof}

\begin{DefSub}
Given a $\T$-algebra $A:\T \rightarrow \C$ for a Lawvere theory $\T$, the \textbf{commutant} $\T^\perp_A$ of $A$ (or of $\T$ with respect to $A$) is defined as the commutant of the associated morphism of theories $A':\T \rightarrow \C_A$, where $\C_A$ is the full finitary theory of $A$ in $\C$ \pbref{def:full_theory}.  By \bref{thm:cmtnt_full_th_alg} we obtain the following equivalent definition:
\end{DefSub}

\begin{CorSub}\label{thm:charn_cmtnt_talg}
The commutant $\T^\perp_A$ of a $\T$-algebra $A:\T \rightarrow \C$ is the full finitary theory of $A$ in the category of $\T$-algebras in $\C$, i.e. $\T^\perp_A \cong (\Alg{\T}_\C)_A$ as theories over $\C_A$.  For suitable choices of finite powers $A^n$ in $\Alg{\T}_\C$, this isomorphism is an identity, so that $\T^\perp_A(n,m) = \Alg{\T}_\C(A^n,A^m)$ for all $n,m \in \NN$.
\end{CorSub}
\begin{proof}
This is verified straightforwardly by applying \bref{thm:cmtnt_full_th_alg} and using the fact that the canonical functor $\iota:\C_A \rightarrow \C$ is fully faithful and preserves finite powers \pbref{def:full_theory}.  With reference to the proof of \bref{thm:cmtnt_full_th_alg}, the pointwise \textit{left} finite powers of $A':\T \rightarrow \C_A$ in $\Alg{\T}_{\C_A}$ induce standard designated finite powers of $A$ in $\Alg{\T}_\C$ by composition with $\iota$.  Employing these in forming $(\Alg{\T}_\C)_A$, we find that $\iota$ induces an isomorphism of theories $(\Alg{\T}_{\C_A})_{A'} \cong (\Alg{\T}_\C)_A$ over $\C_A$, but the convention of \eqref{eq:f_pres_mu} entails that the associated morphisms to $\C_A$ are concrete subtheory embeddings, so the latter isomorphism is an equality of concrete subtheories. 
\end{proof}

\begin{RemSub}
In the case where $\C = \Set$, the commutant $\T^\perp_A$ of a $\T$-algebra $A:\T \rightarrow \Set$ is (by \bref{thm:charn_cmtnt_talg}) an instance of Lawvere's notion of the \textbf{algebraic structure} of a $\Set$-valued functor \cite[III.1]{Law:PhD}.
\end{RemSub}

\begin{ExaSub}[\textbf{The Lawvere theory of $R$-modules}]\label{exa:cmtnt_lth_rmods}
Letting $R$ be a ring or rig, recall that left $R$-modules are the same as normal $\T$-algebras for the Lawvere theory $\T = \Mat_R$ \pbref{exa:law_th_rmods}.  In particular, $R$ is a left $R$-module and so determines a morphism $R:\T \rightarrow \Set_R$ into the full finitary theory $\Set_R$ of $R$ in $\Set$.  Thus regarding $\T$ as a theory over $\Set_R$, its commutant $\T^\perp$ has $\T^\perp(n,m) = \Mod{R}(R^n,R^m)$ by \bref{thm:charn_cmtnt_talg} once we identify $\Mod{R}$ with the isomorphic category $\Alg{\T}^!$.  Moreover, we have an identity-on-homs functor $\T^\perp \rightarrow \Mod{R}$ given on objects by $n \mapsto R^n$.

Let us first observe that $\T^\perp \cong \T^\op$ as categories.  Indeed, as noted in \bref{par:alg_cats} we have a fully faithful functor $y:\T^\op \rightarrowtail \Mod{R}$ sending $n \in \ob\T = \NN$ to the free $R$-module $R^n$.  The constituent isomorphisms $R^{m \times n} = \T(n,m) \cong \Mod{R}(R^m,R^n)$ send each $m \times n$-matrix $u$ to the left $R$-linear map $R^m \rightarrow R^n$ given by \textit{right} multiplication by $u$, i.e., we regard $R^m$ and $R^n$ as sets of \textit{row vectors} so that the associated map $R^{1 \times m} \rightarrow R^{1 \times n}$ is given by $x \mapsto xu$.

Next observe that the identity-on-objects functor $(\Mat_R)^\op \rightarrow \Mat_{R^\op}$ given by transposition is an isomorphism of categories
$$(\Mat_R)^\op \cong \Mat_{R^\op}$$
so that
$$\T^\perp \cong \T^\op \cong \Mat_{R^\op}$$
as categories.  The composite isomorphism $\Mat_{R^\op} \rightarrow \T^\perp$ commutes with the associated morphisms to $\Set_R = \Set_{R^\op}$, as it associates to an $m \times n$-matrix $u$ over $R^\op$ the (left $R$-linear) map $R^n \rightarrow R^m$ given by $x \mapsto ux$ when we regard each $x \in R^n$ as a column vector with entries in $R^\op$.  We thus obtain the following, recalling that left $R^\op$-modules are the same as right $R$-modules.
\end{ExaSub}

\begin{ThmSub}\label{thm:cmtnt_lth_rmods}
Let $R$ be a ring or rig.  Then the commutant $(\Mat_R)^\perp$ with respect to $R$ of the theory of left $R$-modules $\Mat_R$ is the Lawvere theory of right $R$-modules $\Mat_{R^\op}$.  Indeed, we have an isomorphism
$$(\Mat_R)^\perp \cong \Mat_{R^\op}$$
in the category of Lawvere theories over $\Set_R$.
\end{ThmSub}

\section{Saturated and balanced subtheories}\label{sec:sat_bal}

\begin{PropSub}\label{thm:cmtnt_adjn}
Let $\U$ be a Lawvere theory.
\begin{enumerate}
\item Lawvere theories $\T$ and $\sS$ over $\U$ commute if and only if there exists a (necessarily unique) morphism $\sS \rightarrow \T^\perp$ in the category of Lawvere theories over $\U$.
\item There is a unique functor $(-)^\perp:(\Th\slash \U)^\op \rightarrow \Th\slash \U$ sending each theory $\T$ over $\U$ to its commutant $\T^\perp$.  
\item The functor $(-)^\perp$ in 2 is right adjoint to its formal dual $(-)^\perp:\Th\slash \U \rightarrow (\Th\slash \U)^\op$.
\item The adjunction in 3 restricts to a Galois connection on the preordered set $\SubTh(\U)$ of subtheories of $\U$ \pbref{par:cat_th}, i.e., an adjunction between $\SubTh(\U)$ and its opposite.
\end{enumerate}
\end{PropSub}
\begin{proof}
1 follows from \bref{thm:univ_prop_cmtnt}.  Every theory $\T$ over $\U$ commutes with its commutant $\T^\perp$, so by 1 there is a unique morphism $\T \rightarrow \T^{\perp\perp}$ in $\Th\slash\U$.  Given a morphism $M:\sS \rightarrow \T$ in $\Th\slash \U$, we obtain a composite morphism $\sS \xrightarrow{M} \T \rightarrow \T^{\perp\perp}$, so by $1$ $\T^\perp$ and $\sS$ commute and hence there is a unique morphism $\T^\perp \rightarrow \sS^\perp$ in $\Th\slash \U$ and 2 follows.  3 and 4 now follow readily.
\end{proof}

\begin{DefSub}\label{def:sat_bal}\emptybox
\begin{enumerate}
\item Given a Lawvere theory $\T$ over $\U$, we say that
\begin{enumerate}
\item $\T$ is \textbf{saturated} if $\T^{\perp\perp} \cong \T$ in $\Th\slash \U$.
\item $\T$ is \textbf{balanced} if $\T^{\perp} \cong \T$ in $\Th\slash \U$.
\end{enumerate}
\item Given a pair of Lawvere theories $\sS$ and $\T$ over $\U$, we say that $\T$ and $\sS$ are \textbf{mutual commutants} in $\U$ if $\T \cong \sS^\perp$ and $\T^\perp \cong \sS$.
\end{enumerate}
\end{DefSub}

We readily deduce the following:

\begin{PropSub}\label{thm:sat}
\emptybox
\begin{enumerate}
\item Any saturated theory $\T$ over $\U$ is necessarily a subtheory of $\U$.
\item A theory $\T$ over $\U$ is saturated if and only if it is (isomorphic to) a commutant $\sS^\perp$ of some theory $\sS$ over $\U$.
\item Theories $\T$ and $\sS$ over $\U$ are mutual commutants if and only if $\T$ is saturated and $\sS$ is its commutant.
\item A subtheory $\T$ of $\U$ is commutative (as a Lawvere theory) if and only if $\T$ is contained in its commutant.
\end{enumerate}
\end{PropSub}

\begin{CorSub}\label{thm:bal_impl_comm_sat}
Every balanced theory over $\U$ is a commutative, saturated subtheory of $\U$.
\end{CorSub}
\newpage
\begin{ThmSub}\label{thm:th_lr_rmods_mutual_cmtnts}
Let $R$ be a ring or rig.
\begin{enumerate}
\item The Lawvere theories $\Mat_R$ and $\Mat_{R^\op}$ of left and right $R$-modules (respectively) are mutual commutants in the full finitary theory $\Set_R$ of $R$ in $\Set$.
\item The Lawvere theory of left $R$-modules $\Mat_R$ is a saturated subtheory of $\Set_R$.
\item The subtheory $\Mat_R \hookrightarrow \Set_R$ is balanced if and only if $R$ is commutative.
\end{enumerate}
\end{ThmSub}
\begin{proof}
1 is obtained by two applications of \bref{thm:cmtnt_lth_rmods}, and 2 then follows immediately.  If $R$ is commutative then $\Mat_R = \Mat_{R^\op}$ as theories over $\Set_R$, so $\Mat_R$ is balanced, by 1.  Conversely, if $\Mat_R$ is balanced over $\Set_R$, then $\Mat_R$ is commutative by \bref{thm:bal_impl_comm_sat}, so $R$ is commutative by \eqref{eq:lth_rmods_cmtv_iff_r_cmtv}.
\end{proof}

\begin{ExaSub}\label{exa:lth_slats_bal}
By \bref{thm:th_lr_rmods_mutual_cmtnts}, the Lawvere theory of semilattices $\Mat_2$ \pbref{exa:slats} is a balanced subtheory of the Lawvere theory of Boolean algebras $\Set_2$ \pbref{exa:bool}.
\end{ExaSub}

\begin{ExaSub}[\textbf{A non-saturated subtheory}]
Let $k$ be an infinite integral domain, and consider the Lawvere theory of commutative $k$-algebras $\T$ \pbref{exa:lth_ckalgs}.  The domain $k$ itself is a commutative $k$-algebra and so determines a morphism of Lawvere theories $\kappa:\T \rightarrow \Set_k$ into the full finitary theory $\Set_k$ of $k$ in $\Set$ \pbref{def:full_theory}.  For each natural number $n$, the associated component $\kappa_{n,1}:\T(n,1) \rightarrow \Set_k(n,1)$ is the mapping $k[x_1,...,x_n] \rightarrow \Set(k^n,k)$ that sends a polynomial $f$ to the polynomial function $k^n \rightarrow k$ determined by $f$.  Since $k$ is an infinite integral domain, this mapping $\kappa_{n,1}$ is injective (e.g. by \cite[III.4, Thm. 7]{BirMac}), so $\kappa$ presents $\T$ as a subtheory of $\Set_k$.  The commutant of this subtheory $\T$ is the subtheory $\T^\perp \hookrightarrow \Set_k$ in which $\T^\perp(n,1) = \CAlg{k}(k^n,k)$ is the set of all $k$-algebra homomorphisms $\varphi:k^n \rightarrow k$.  But any such homomorphism $\varphi$ is $k$-linear and so is a linear combination $\varphi = \sum_{i = 1}^n c_i\pi_i$ of the projections $\pi_i:k^n \rightarrow k$ $(i = 1,...,n)$, where $c_i = \varphi(b_i)$ is the image of the $i$-th standard basis vector $b_i$ for $k^n$.  We also know that $1 = \varphi(1) = \sum_i c_i$.  Hence since $i \neq j$ implies $b_ib_j = 0$ in $k^n$ and $\varphi$ preserves multiplication, it follows that $\varphi = \pi_i$ for a unique $i$.  Therefore $\T^\perp(n,1)$ is just the set of all $n$ projections $k^n \rightarrow k$, and (since $k$ has at least two elements) it follows that $\T^\perp$ is isomorphic to the initial Lawvere theory $\FinCard^\op$.  Therefore $\T^\perp \hookrightarrow \Set_k$ is central (by \bref{exa:uniq_mor_from_initial_th_is_central}) and hence $\T^{\perp\perp} = \Set_k$, but $\Set_k \not\cong \T$ by a cardinality argument: $\text{\#}\Set_k(1,1) = (\text{\#} k)^{\text{\#}k} \gt 2^{\text{\#}k} > \text{\#}k = \text{\#}k[x] = \text{\#}\T(1,1)$.
\end{ExaSub}

\section{The theories of affine and convex spaces as commutants}\label{sec:th_affsp_as_cmtnt}

Let $R$ be ring or, more generally, a rig.  

\begin{ParSub}[\textbf{Pointed $R$-modules}]\label{par:ptd_rmod}
By definition, a \textbf{pointed (left)} $R$-\textbf{module} is a (left) $R$-module $M$ equipped with an arbitrary chosen element $* \in M$.  Pointed $R$-modules are objects of a category $\Mod{R}^*$ in which the morphisms are $R$-module homomorphisms that preserve the chosen points $*$.  By \bref{par:var_algs}, $\Mod{R}^*$ is strictly finitary-algebraic over $\Set$.  Given a natural number $n$, the free pointed $R$-module on $n$-generators is the free $R$-module on $1 + n$ generators $R^{1 + n}$.  Indeed, writing the successive standard basis vectors for $R^{1 + n}$ as $\gamma_0,\gamma_1,...,\gamma_n$, we find that $R^{1 + n}$ is a free pointed $R$-module on the $n$ generators $\gamma_1,...,\gamma_n$ when we take $* = \gamma_0 = (1,0,0,...,0)$.  By \bref{par:alg_cats}, $\Mod{R}^*$ is therefore isomorphic to the category of normal $\T$-algebras for a Lawvere theory $\T = \Mat_R^*$ with $\Mat_R^*(n,m) = (R^{1 + n})^m = R^{m \times (1+n)}$.

The notion of \textbf{pointed right} $R$-\textbf{module} is defined similarly, so that pointed right $R$-modules are the same as pointed left $R^\op$-modules, equivalently, normal $\Mat_{R^\op}^*$-algebras.  Given a pointed right $R$-module $M$, we shall now record a detailed description of the corresponding normal $\Mat_{R^\op}^*$-algebra $\underline{M}:\Mat_{R^\op}^* \rightarrow \Set$ for use in the sequel.  By \bref{par:alg_cats}, $\underline{M}$ has the same carrier as $M$ and associates to each $w \in \Mat^*_{R^\op}(n,1) = R^{1 + n}$ the mapping $\Phi^M_w:M^n \rightarrow M$ defined as follows.  Recalling that $R^{1 + n}$ is a free pointed right $R$-module on the $n$ generators $\gamma_1,...,\gamma_n \in R^{1 + n}$, each $n$-tuple $x = (x_1,...x_n) \in M^n$ induces a unique morphism of pointed right $R$-modules $x^\sharp:R^{1 + n} \rightarrow M$ with $x^\sharp(\gamma_i) = x_i$ $(i = 1,...,n)$, and the associated mapping $\Phi^M_w:M^n \rightarrow M$ is given by
$$\Phi^M_w(x) = x^\sharp(w) = * \cdot w_0 + \sum_{i = 1}^n x_iw_i$$
where $* \in M$ is the designated point and $w = (w_0,w_1,...,w_n) \in R^{1+n}$.

In particular, we can consider $R$ itself as a pointed right $R$-module with chosen point $1 \in R$.  Therefore $R$ is the carrier of a normal $\Mat_{R^\op}^*$-algebra, and we can thus consider $\Mat_{R^\op}^*$ as a Lawvere theory over the full finitary theory $\Set_R$ of $R$ in $\Set$.  Explicitly, we have a morphism
\begin{equation}\label{eq:can_mor_on_th_pted_rmods}\Phi^R:\Mat_{R^\op}^* \rightarrow \Set_R\end{equation}
sending each $w = (w_0,...,w_n) \in R^{1 + n}$ to the mapping $\Phi^R_w:R^n \rightarrow R$ given by
\begin{equation}\label{eq:aff_map_det_by_pted_rmod_op}\Phi^R_w(x) = w_0 + \sum_{i = 1}^n x_iw_i\;.\end{equation}
\end{ParSub}

\begin{ThmSub}\label{thm:taffr_as_cmtnt}
The Lawvere theory of left $R$-affine spaces $\Mat_R^\aff$ is the commutant with respect to $R$ of the theory of pointed right $R$-modules $\Mat_{R^\op}^*$.  Indeed,
$$\Mat_R^\aff \cong (\Mat_{R^\op}^*)^\perp$$
as Lawvere theories over $\Set_R$ when $\Mat_R^\aff$ is equipped with the morphism $\Mat_R^\aff \rightarrow \Set_R$ determined by the left $R$-affine space $R$.
\end{ThmSub}
\begin{proof}
By \bref{thm:th_lr_rmods_mutual_cmtnts} we know that $\Mat_R \cong (\Mat_{R^\op})^\perp$ over $\Set_R$, so the theory $\Mat_R^\aff$ is isomorphic to the affine core $\T^\aff$ of $\T = (\Mat_{R^\op})^\perp$.  This yields an isomorphism $\Mat_R^\aff \cong \T^\aff$ as theories over $\Set_R$ since the inclusion $\Mat_R^\aff \hookrightarrow \Mat_R$ is a morphism over $\Set_R$.  Recall that $\T$ is the concrete subtheory of $\Set_R$ consisting of all right $R$-linear maps $\varphi:R^n \rightarrow R^m$.  Therefore the affine core $\T^\aff$ of $\T$ is the subtheory of $\Set_R$ consisting of all right $R$-linear maps $\varphi:R^n \rightarrow R^m$ that commute with the `diagonal' maps $(1,...,1):R^1 \rightarrow R^j$ with $j = n,m$.  But these are precisely the homomorphisms of pointed right $R$-modules $\varphi:R^n \rightarrow R^m$, where we regard the powers $R^j$ of $R$ as pointed right $R$-modules with chosen point $1 = (1,...,1) \in R^j$.  With this convention $R^j$ is the $j$-th power of $R = R^1$ in the category of pointed right $R$-modules, so the subtheory $\T^\aff$ of $\Set_R$ is precisely the commutant with respect to $R$ of the theory of pointed right $R$-modules $\Mat_{R^\op}^*$.  Hence $\Mat_R^\aff \cong \T^\aff = (\Mat_{R^\op}^*)^\perp$ over $\Set_R$.
\end{proof}

\begin{CorSub}\label{thm:th_raff_sat}\emptybox
\begin{enumerate}
\item The Lawvere theory of left $R$-affine spaces $\Mat_R^\aff$ is a saturated subtheory of the full finitary theory $\Set_R$ of $R$ in $\Set$.
\item When $R$ is commutative, $\Mat_R^\aff$ is a commutative, saturated subtheory of $\Set_R$.
\item If $R \neq 0$, then the subtheory $\Mat_R^\aff \hookrightarrow \Set_R$ is not balanced.
\end{enumerate}
\end{CorSub}
\begin{proof}
1 and 2 follow immediately from \bref{thm:taffr_as_cmtnt}, \bref{thm:sat}, and \bref{exa:commutative_lths}.  For 3, suppose that $R \neq 0$, and fix some nonzero element $r$ of $R$.  Considering $\Mat_R^\aff$ and $\Mat_{R^\op}^*$ as Lawvere theories over $\Set_R$, we know by \bref{thm:taffr_as_cmtnt} that $\Mat_R^\aff \cong (\Mat_{R^\op}^*)^\perp$, so in order to show that $\Mat_R^\aff$ is not balanced it suffices to show that $(\Mat_R^\aff)^\perp \not\leq (\Mat_{R^\op}^*)^\perp$ as subtheories of $\Set_R$.  But the constant map $R^n \rightarrow R$ with value $r$ is left $R$-affine and is not right $R$-linear, so this constant map is an element of $(\Mat_R^\aff)^\perp(n,1) = \Aff{R}(R^n,R)$ but is not an element of $(\Mat_{R^\op}^*)^\perp(n,1) = \Mod{R^\op}^*(R^n,R)$.
\end{proof}

\begin{ExaSub}\label{exa:th_unb_jslats_sat_subth}
By \bref{thm:th_raff_sat}, the theory $\Mat_2^\aff$ of unbounded join semilattices (equivalently, $2$-affine spaces, \bref{exa:unb_jslats}) is a non-balanced, commutative, saturated subtheory of the theory $\Set_2$ of Boolean algebras \pbref{exa:bool}.
\end{ExaSub}

\section{The commutant of the theory of unbounded semilattices}\label{sec:cmt_th_unb_slats}

Given a ring or rig $R$, we showed in the previous section that the theory of left $R$-affine spaces $\Mat_R^\aff$ is the commutant of the theory $\Mat_{R^\op}^*$ of pointed right $R$-modules, considered as a theory over $\Set_R$.  In the remainder of the paper, we shall examine the commutant of $\Mat_R^\aff$ over $\Set_R$, and in particular we ask whether this commutant is $\Mat_{R^\op}^*$, i.e. whether the theories $\Mat_R^\aff$ and $\Mat_{R^\op}^*$ are mutual commutants in the full finitary theory of $R$ in $\Set$.  We shall later show that this is indeed the case for all rings \pbref{thm:cmtnt_of_th_of_raff_sp} and also for many rigs that are not rings \pbref{thm:suff_conds}.  But first we will examine a notable example of a rig for which this is \textit{not} the case.

Writing simply $2$ to denote the rig $(2,\vee,0,\wedge,1)$, recall that $2$-affine spaces are the same as unbounded join semilattices \pbref{exa:unb_jslats}, equivalently, idempotent commutative semigroups, whereas $2$-modules are (bounded) join semilattices \pbref{exa:slats}.  By \bref{exa:th_unb_jslats_sat_subth}, we know that the theory of unbounded join semilattices is a saturated subtheory of the theory $\Set_2$ of Boolean algebras, and in the present section we characterize its commutant.  

\begin{ParSub}[\textbf{Join semilattices with top element}]\label{par:jslatwtop}

Let $\SLat_{\scriptscriptstyle \vee\top}$ denote the category whose objects are join semilattices with a top element and whose morphisms are homomorphisms of join semilattices that preserve the top element.  Join semilattices with a top element can be described equivalently as idempotent commutative monoids $S$ with an additional constant $\top$ satisfying a single additional equation $s \cdot \top = \top$ $(s \in S)$, so $\SLat_{\scriptscriptstyle \vee\top}$ is a variety of finitary algebras and hence (by \bref{par:var_algs}) is isomorphic to the category of normal $\T$-algebras for a Lawvere theory $\T = \T_{\scriptscriptstyle \bot\vee\top}$.

In order to obtain a description of $\T_{\scriptscriptstyle \bot\vee\top}$, observe that for each finite cardinal $n$, the free join-semilattice-with-top-element $F(n)$ on $n$ generators can be obtained by artificially adjoining a new top element $\top$ to the free join semilattice $2^n = \sP(n)$ on $n$ generators \pbref{exa:slats}, where as generators we take the standard basis vectors $b_1,...,b_n \in 2^n$, i.e., the singleton subsets of $n$.  More precisely, we let $F(n) := \sP(n) + \{\top\}$ and observe that $F(n)$ then carries a unique join semilattice structure such that $\top$ is a top element of $F(n)$ and such that the inclusion $\sP(n) \hookrightarrow F(n)$ is a homomorphism of join semilattices.  Now let $S$ be a join semilattice with a top element, and let $x = (x_1,...,x_n) \in S^n$.  The universal property of $\sP(n)$ as a join semilattice now clearly entails that there is a unique morphism $x^\sharp:F(n) \rightarrow S$ in $\SLat_{\scriptscriptstyle \vee\top}$ with $x^\sharp(b_i) = x_i$ for all $i$.

Hence by \bref{par:alg_cats} the theory $\T_{\scriptscriptstyle \bot\vee\top}$ has
$$\T_{\scriptscriptstyle \bot\vee\top}(n,1) = F(n) = 2^n + \{\top\}\;.$$
The join semilattice $2 = (2,\vee,0)$ has top element $1$ and so (by \bref{def:full_theory}) determines a morphism of theories
\begin{equation}\label{eq:can_mor_on_th_jslats_w_top}\underline{2}:\T_{\scriptscriptstyle \bot\vee\top} \rightarrow \Set_2\end{equation}
by means of which $\T_{\scriptscriptstyle \bot\vee\top}$ can be considered as a theory over the full finitary theory $\Set_2$ of $2$ in $\Set$.  Recalling that $\Set_2$ is the theory of Boolean algebras \pbref{exa:bool}, we shall prove the following:
\end{ParSub}

\begin{ThmSub}\label{thm:th_uslats_slatswtop_cmtnts}
The following Lawvere theories are mutual commutants in the theory of Boolean algebras:
\begin{enumerate}
\item The theory $\Mat_2^\aff$ of unbounded join semilattices (equivalently, $2$-affine spaces);
\item the theory $\T_{\scriptscriptstyle \bot\vee\top}$ of join semilattices with top element.
\end{enumerate}
\end{ThmSub}
\begin{proof}
By \bref{thm:taffr_as_cmtnt} we know that $\Mat_2^\aff \cong (\Mat_2^*)^\perp$ as theories over $\Set_2$, where $\Mat_2^*$ is the theory of pointed $2$-modules, equivalently, pointed join semilattices.  Explicitly, $(\Mat_2^*)^\perp$ is the concrete subtheory of $\Set_2$ consisting of all homomorphisms of join semilattices $\varphi:2^n \rightarrow 2^m$ that preserve the designated points $1 = (1,...,1) \in 2^j$ with $j = n,m$.  But the join semilattice $2$ has top element $1 \in 2$, and the finite powers $2^j$ of $2$ in $\SLat_{\scriptscriptstyle \vee\top}$ have top element $1 = (1,...,1) \in 2^j$, so $(\Mat_2^*)^\perp$ consists of all morphisms $\varphi:2^n \rightarrow 2^m$ in $\SLat_{\scriptscriptstyle \vee\top}$.  Hence
$$\Mat_2^\aff \cong (\Mat_2^*)^\perp = (\T_{\scriptscriptstyle \bot\vee\top})^\perp$$
as theories over $\Set_2$.

In particular, $\T_{\scriptscriptstyle \bot\vee\top}$ and $\Mat_2^\aff$ commute over $\Set_2$, so the morphism \eqref{eq:can_mor_on_th_jslats_w_top} factors through the commutant $(\Mat_2^\aff)^\perp \hookrightarrow \Set_2$.  We therefore have a morphism 
$$\underline{2}:\T_{\scriptscriptstyle \bot\vee\top} \rightarrow (\Mat_2^\aff)^\perp$$
that sends each $\omega \in \T_{\scriptscriptstyle \bot\vee\top}(n,1) = F(n) = 2^n + \{\top\}$ to an element $\underline{2}_\omega \in (\Mat_2^\aff)^\perp(n,1) = \USLat_\vee(2^n,2)$, i.e. a homomorphism of unbounded join semilattices $\underline{2}_\omega:2^n \rightarrow 2$.  By \bref{par:alg_cats}, this map $\underline{2}_\omega$ sends each $x = (x_1,...,x_n) \in 2^n$ to $\underline{2}_\omega(x) = x^\sharp(\omega)$, recalling from \bref{par:jslatwtop} that  $x^\sharp:F(n) \rightarrow 2$ is the unique morphism in $\SLat_{\scriptscriptstyle \vee\top}$ such that $x^\sharp(b_i) = x_i$ for all $i$.  Explicitly,
$$\underline{2}_\omega(x) = x^\sharp(\omega) = \begin{cases}
  \bigvee_{i = 1}^n (\omega_i \wedge x_i) & \text{if $\omega = (\omega_1,...,\omega_n) \in 2^n$}\\
  1 & \text{if $\omega = \top$}
  \end{cases}.
$$
Hence it suffices to show that each homomorphism of unbounded join semilattices $\varphi:2^n \rightarrow 2$ is equal to $\underline{2}_\omega$ for a unique $\omega \in 2^n + \{\top\}$.  If $\varphi(0) = 0$ then $\varphi$ is a homomorphism of join semilattices (or $2$-modules) so since the theory of $2$-modules $\Mat_2$ is a balanced subtheory of $\Set_2$ \pbref{exa:lth_slats_bal} it follows that $\varphi$ is one of the $2$-module operations \eqref{eq:linear_comb_ops} carried by $2$.  More precisely, there is a unique row vector $\omega \in 2^n = \Mat_2(n,1)$ such that $\varphi:2^n \rightarrow 2$ is given by $x \mapsto \bigvee_{i = 1}^n (\omega_i \wedge x_i)$, and $\omega$ is then the unique element of $2^n + \{\top\}$ with $\underline{2}_\omega = \varphi$.  On the other hand, if $\varphi(0) \neq 0$ then $\varphi(0) = 1$, but $\varphi$ preserves binary joins and hence is monotone, so $\varphi:2^n \rightarrow 2$ must be the constant map with value $1$ and hence we find that $\omega = \top$ is the unique element $\omega \in 2^n + \{\top\}$ with $\underline{2}_\omega = \varphi$.
\end{proof}

\section{The commutant of the theory of affine spaces over a ring}\label{sec:cmtnt_of_th_cvx_sp}

The remainder of the paper is devoted to showing that the theories of left $R$-affine spaces and pointed right $R$-modules are mutual commutants in $\Set_R$ for many rigs $R$.  In the present section, we show that this holds for all rings.

\begin{ParSub}\label{par:can_mor}
Given an arbitrary rig $R$, we will consider $\Mat^*_{R^\op}$ and $\Mat^\aff_R$ as theories over $\Set_R$, via the morphisms to $\Set_R$ determined by $R$ (\bref{par:ptd_rmod}, \bref{thm:taffr_as_cmtnt}).  By \bref{thm:taffr_as_cmtnt} we know that these theories over $\Set_R$ commute, so the canonical morphism $\Phi^R:\Mat_{R^\op}^* \rightarrow \Set_R$ factors through the inclusion $(\Mat^\aff_R)^\perp \hookrightarrow \Set_R$ via a unique morphism
\begin{equation}\label{eq:can_mor}\Phi^R:\Mat_{R^\op}^* \rightarrow (\Mat^\aff_R)^\perp\end{equation}
for which we use the same notation $\Phi^R$.  By \bref{par:ptd_rmod}, the components $\Mat_{R^\op}^*(n,1) \rightarrow (\Mat^\aff_R)^\perp(n,1)$ of this morphism are the maps
\begin{equation}\label{eq:can_map}\Phi^R_{(-)}:R^{1 + n} \rightarrow \Aff{R}(R^n,R)\end{equation}
that send each element $w = (w_0,...,w_n)$ of $R^{1 + n}$ to the $R$-affine map $\Phi^R_w:R^n \rightarrow R$ given by
$$\Phi^R_w(x) = w_0 + \sum_{i = 1}^n x_iw_i\;.$$
\end{ParSub}

\begin{ThmSub}\label{thm:cmtnt_of_th_of_raff_sp}
Given a ring $R$, the commutant $(\Mat^\aff_R)^\perp$ with respect to $R$ of the theory $\Mat^\aff_R$ of left $R$-affine spaces is the theory $\Mat_{R^\op}^*$ of pointed right $R$-modules.  Indeed, the morphism \eqref{eq:can_mor} is an isomorphism
$$\Mat_{R^\op}^* \;\;\cong\;\; (\Mat^\aff_R)^\perp$$
in the category of Lawvere theories over the full finitary theory $\Set_R$ of $R$ in $\Set$.
\end{ThmSub}
\begin{proof}
Let $n \in \NN$.  For each left $R$-affine map $\psi:R^n \rightarrow R$ let us denote by
$$w^\psi = (w_0^\psi,w_1^\psi,...,w_n^\psi)$$
the element of $R^{1 + n}$ defined by
$$w^\psi_0 = \psi(0),\;\;\;\;\;\;w^\psi_i = \psi(b_i) - \psi(0)\;\;\;(i = 1,...,n)$$
where $b_i = (b_{i1},...,b_{in}) \in R^n$ is the $i$-th standard basis vector (with $b_{ii} = 1$ and $b_{ij} = 0$ for $j \neq i$).  This defines a map
$$w^{(-)}:\Aff{R}(R^n,R) \rightarrow R^{1 + n}$$
which we claim is inverse to the mapping $\Phi^R_{(-)}$ of \eqref{eq:can_map}.  Indeed, it is easy to see that $w^{(-)}$ is a retraction of $\Phi^R_{(-)}$, so it suffices to show that for any left $R$-affine map $\psi:R^n \rightarrow R$, if we let $w = w^\psi$ then $\Phi^R_w:R^n \rightarrow R$ is exactly $\psi$.  To this end, observe that any element $x = (x_1,...,x_n)$ of $R^n$ can be expressed as a left $R$-affine combination
$$x = \left(1 - \sum_{i = 1}^n x_i\right)0 \;\;+\;\; \sum_{i = 1}^n x_ib_i$$
of the elements $0,b_1,...,b_n$ of $R^n$, so $\psi$ necessarily sends $x$ to
$$
\begin{array}{ccccc}
\displaystyle{\psi(x)} & = & \displaystyle{\left(1 - \sum_{i = 1}^n x_i\right)\psi(0) \;\;+\;\; \sum_{i = 1}^n x_i\psi(b_i)} & &\\
           & = & \displaystyle{\psi(0) \;\;+\;\;\sum_{i = 1}^n x_i(\psi(b_i) - \psi(0))} & = &  \displaystyle{\Phi^R_w(x)}\;.
\end{array}
$$
\end{proof}

By \bref{thm:taffr_as_cmtnt}, we obtain the following:

\begin{CorSub}\label{thm:raff_prmod_cmtnts_for_ringr}
For a ring $R$, the Lawvere theory of left $R$-affine spaces and the Lawvere theory of pointed right $R$-modules are mutual commutants in the full finitary theory of $R$ in $\Set$.
\end{CorSub}

\section{The commutant of the theory of convex spaces for a preordered ring}\label{sec:cmtnt_th_cvx_sp}

Having shown that theories of left $R$-affine spaces and pointed right $R$-modules for a ring $R$ are mutual commutants in $\Set_R$ \pbref{thm:raff_prmod_cmtnts_for_ringr}, we now set ourselves to the task of widening the applicability of this result to include certain rigs that are not rings.  In particular, we will focus on \textit{additively cancellative} rigs $S$ \pbref{par:pr_rings}, which are precisely the positive parts $S = R_+$ of preordered rings $R$.  Our study of the commutant of the theory of $R_+$-affine spaces (which we also call $R$-convex spaces) begins with the following observations:

\begin{LemSub}\label{thm:aff_extn_prop}
Let $R$ be a preordered ring, and let $\varphi:R^n_+ \rightarrow R_+$ be a (left) $R_+$-affine map with $n \in \NN$.  Then $\varphi$ is a restriction of at most one $R$-affine map $\psi:R^n \rightarrow R$.  Further, the following are equivalent:
\begin{enumerate}
\item $\varphi$ is a restriction of some $R$-affine map $\psi:R^n \rightarrow R$.
\item $\varphi$ is a restriction of the $R$-affine map $\Phi^R_{w^\varphi}:R^n \rightarrow R$ \pbref{par:can_mor} determined by
$$w^\varphi = (\varphi(0),\varphi(b_1) - \varphi(0),...,\varphi(b_n) - \varphi(0)) \in R^{1 + n}\;,$$
where $b_i \in R_+^n$ is the $i$-th standard basis vector.
\item For every element $x = (x_1,...,x_n) \in R_+^n$,
\begin{equation}\label{eq:affine_fnl_on_sn}\varphi(x) = \varphi(0) + \sum_{i = 1}^n x_i(\varphi(b_i) - \varphi(0))\end{equation}
in $R$.
\end{enumerate}
\end{LemSub}
\begin{proof}
It is clear from the proof of \bref{thm:cmtnt_of_th_of_raff_sp} that an $R$-affine map $\psi:R^n \rightarrow R$ is uniquely determined by its values on the elements $0,b_1,...,b_n$ of $R^n$ and so is uniquely determined by its restriction to $R_+^n$.  The equivalence of 2 and 3 is immediate from the definitions, and clearly 2 implies 1.  If 1 holds, then $\varphi$, $\psi$, and $\Phi^R_{w^\varphi}$ all agree on the elements $0,b_1,...,b_n$ of $R^n$, so $\psi$ and $\Phi^R_{w^\varphi}$ are equal and hence 2 holds.
\end{proof}

\begin{DefSub}
We say that a preordered ring $R$ has the \textbf{affine extension property} if every $R_+$-affine map $\varphi:R_+^n \rightarrow R_+$ $(n \in \NN)$ satisfies the equivalent conditions of \bref{thm:aff_extn_prop}.
\end{DefSub}

We shall find that the affine extension property is a necessary condition for the Lawvere theories of left $R$-convex spaces and pointed right $R_+$-modules to be mutual commutants in $\Set_{R_+}$.  In order to obtain necessary and sufficient conditions, we will also need to impose a certain weakening of the \textit{archimedean property}.  The familiar archimedean property for totally ordered fields has been generalized to the context of partially ordered rings and abelian groups in a number of slightly different ways in the literature.  We shall now recall one common definition, which appears for example in \cite{Kea}, and then proceed to define the weaker property that we shall require.

\begin{DefSub}
Let $R$ be a preordered ring.
\begin{enumerate}
\item We say that $R$ is \textbf{archimedean} provided that for each element $r \in R$, if $\{nr \mid n \in \NN\}$ has an upper bound in $R$ then $r \lt 0$.
\item Given an element $r \in R$, we call the subset $\{sr \mid s \in R_+\} \subseteq R$ the \textbf{(left) ray} of $r$.
\item We say that $R$ is \textbf{(left) auto-archimedean} provided that for every element $r$ of $R$, if the ray of $r$ has an upper bound in $R$, then $r \lt 0$.
\end{enumerate}
\end{DefSub}

\begin{RemSub}
Observe that an archimedean preordered ring $R$ is necessarily auto-archimedean.
\end{RemSub}

\begin{RemSub}\label{rem:rephr_def_autoarch}
By considering the additive inverse $-r$ of each $r \in R$, we find that a preordered ring $R$ is auto-archimedean precisely when for every $r \in R$, if the ray of $r$ has a lower bound in $R$, then $r \gt 0$.
\end{RemSub}

\begin{ExaSub}[\textbf{The real numbers}]
The totally ordered ring of real numbers $\RR$ is archimedean and hence auto-archimedean.  It is well-known that $\RR$ also has the affine extension property, and in \bref{thm:suff_conds_for_aff_ext_prop} we will prove a more general result that entails this fact.
\end{ExaSub}

\begin{ExaSub}[\textbf{The integers}]\label{exa:ints}
The ring of integers $\ZZ$ under the natural order is archimedean, but $\ZZ$ does not have the affine extension property.  Indeed, $\ZZ_+ = \NN$ and \textit{every} mapping $\NN^n \rightarrow \NN$ is $\NN$-affine $(n \in \NN)$, so for example the mapping $\varphi = 2^{(-)}:\NN \rightarrow \NN$ is $\NN$-affine but does not extend to a $\ZZ$-affine map $\ZZ \rightarrow \ZZ$.
\end{ExaSub}

\begin{LemSub}\label{thm:ray_bdd}
For a preordered ring $R$, the following are equivalent:
\begin{enumerate}
\item $R$ is auto-archimedean.
\item For all $r,c \in R$, if the $R$-affine map $R \rightarrow R$ given by $x \mapsto c + xr$ maps $R_+$ into $R_+$, then $r,c \in R_+$.
\item For every $n \in \NN$ and every $w \in R^{1 + n}$, if the associated $R$-affine map $\Phi^R_w:R^n \rightarrow R$ maps $R^n_+$ into $R_+$, then $w \in R_+^{1 + n}$.
\end{enumerate}
\end{LemSub}
\begin{proof}
By \bref{rem:rephr_def_autoarch}, 1 holds iff
$$\forall r,b \in R \;:\; (\forall s \in R_+ \;:\; b \lt sr)\;\Rightarrow\; r \in R_+\;,$$
and (by taking $c = -b$) we find that this holds iff
\begin{equation}\forall r,c \in R \;:\; (\forall s \in R_+ \;:\; c + sr \in R_+)\;\Rightarrow\; r \in R_+\;.\end{equation}
Hence 1 holds iff for all $r,c \in R$, if the map $R \rightarrow R$ given by $x \mapsto c + xr$ maps $R_+$ into $R_+$, then $r \in R_+$.  But if $c + xr \in R_+$ for all $x \in R_+$ then we necessarily have $c \in R_+$, so 1 is equivalent to 2.  Also, 3 clearly implies 2.  Assume 2 holds, and suppose that $\Phi^R_w:R^n \rightarrow R$ maps $R^n_+$ into $R_+$, where $w = (w_0,w_1,...,w_n)$ is an element of $R^{1 + n}$.  For each $i = 1,...,n$ we have a mapping $(-) \cdot b_i:R \rightarrow R^n$ given by $r \mapsto rb_i$ where $b_i \in R^n$ is the $i$-th standard basis vector, and the composite
$$\varphi_i = \left(R \xrightarrow{(-) \cdot b_i} R^n \xrightarrow{\Phi^R_w} R\right)$$
is given by $\varphi_i(x) = \Phi^R_w(xb_i) = w_0 + xw_i$.  But since $\Phi^R_w$ maps $R^n_+$ into $R_+$ it follows that $\varphi_i$ maps $R_+$ into $R_+$, so by 2 we deduce that $w_0,w_i \in R_+$.  Therefore if $n \gt 1$ then $w \in R_+^{1 + n}$, but if $n = 0$ then it is readily seen that $w = w_0 \in R_+^{1}$. Hence 3 holds.
\end{proof}

\begin{LemSub}\label{lem:can_mor_inj}
For each $n \in \NN$, the composite
$$R_+^{1 + n} \xrightarrow{\Phi^{R_+}_{(-)}} \Aff{R_+}(R_+^n,R_+) \xrightarrow{w^{(-)}} R^{1 + n}$$
is the inclusion $R_+^{1 + n} \hookrightarrow R^{1 + n}$, where $\Phi^{R_+}_{(-)}$ is the mapping \eqref{eq:can_map} associated to the rig $R_+$ and $w^{(-)}$ is the mapping sending an $R_+$-affine map $\varphi$ to the vector $w^\varphi \in R^{1 + n}$ of \bref{thm:aff_extn_prop}.  Equivalently, if $\varphi = \Phi^{R_+}_w$ then $w = w^\varphi$.  In particular, $\Phi^{R_+}_{(-)}$ is injective.
\end{LemSub}
\begin{proof}
Any $R_+$-affine map of the form $\varphi = \Phi^{R_+}_w$ is a restriction of the $R$-affine map $\Phi^R_w:R^n \rightarrow R$ and and so by \bref{thm:aff_extn_prop} we deduce that $\Phi^R_w = \Phi^R_{w^\varphi}$, whence $w = w^\varphi$ by \bref{thm:cmtnt_of_th_of_raff_sp}.
\end{proof}

\begin{ThmSub}\label{thm:cvx_cmtnt_charn}
The following are equivalent for a preordered ring $R$:
\begin{enumerate}
\item $R$ is auto-archimedean and has the affine extension property.
\item The commutant with respect to $R_+$ of the theory $\Mat_{R_+}^\aff$ of left $R$-convex spaces (i.e., left $R_+$-affine spaces) is isomorphic to the theory of pointed right $R_+$-modules $\Mat_{R_+^\op}^*$, where both theories are considered as Lawvere theories over $\Set_{R_+}$.
\end{enumerate}
\end{ThmSub}
\begin{proof}
By \bref{par:can_mor}, we know that there is a unique morphism $\Mat_{R_+^\op}^* \rightarrow (\Mat_{R_+}^\aff)^\perp$ in $\Th\slash \Set_{R_+}$, namely $\Phi^{R_+}$, so 2 holds iff the mapping 
$$\Phi^{R_+}_{(-)}:R_+^{1 + n} \rightarrow \Aff{R_+}(R^n_+,R_+)$$
is a bijection for each $n \in \NN$.  But by \bref{lem:can_mor_inj} this mapping is necessarily injective and has as its image the set of all $R_+$-affine maps $\varphi:R_+^n \rightarrow R_+$ such that the element $w^\varphi$ of $R^{1 + n}$ lies in $R^{1 + n}_+$ and $\varphi = \Phi^{R_+}_{w^\varphi}$.  Hence 2 holds iff
\begin{enumerate}
\item[$2'$.]  For each $R_+$-affine map $\varphi:R^n_+ \rightarrow R_+$, $w^\varphi \in R_+^{1 + n}$ and $\varphi = \Phi^{R_+}_{w^\varphi}$.
\end{enumerate}
This is clearly equivalent to the following:
\begin{enumerate}
\item[$2''$.]  For each $R_+$-affine map $\varphi:R^n_+ \rightarrow R_+$, $w^\varphi \in R_+^{1 + n}$ and $\varphi$ is a restriction of $\Phi^R_{w^\varphi}:R^n \rightarrow R$.
\end{enumerate}
If 1 holds then this holds, since if $\varphi:R^n_+ \rightarrow R_+$ is $R_+$-affine then the affine extension property entails that $\varphi$ is a restriction of $\Phi^R_{w^\varphi}$ \pbref{thm:aff_extn_prop}, and since $R$ is auto-archimedean \bref{thm:ray_bdd} entails that $w^\varphi \in R_+^{1 + n}$.  Conversely, suppose that $2''$ holds.   Then $R$ clearly has the affine extension property, and we show by way of condition 3 in \bref{thm:ray_bdd} that $R$ is auto-archimedean.  Suppose that $\Phi^R_w:R^n \rightarrow R$ maps $R_+^n$ into $R_+$.  Then its restriction $\varphi:R^n_+ \rightarrow R_+$ is $R_+$-affine, and by $2''$ we know that $\varphi$ is also a restriction of $\Phi^R_{w^\varphi}$.  Hence by \bref{thm:aff_extn_prop} we deduce that $\Phi^R_w =\Phi^R_{w^\varphi}$, so by \bref{thm:cmtnt_of_th_of_raff_sp} we find that $w = w^\varphi$, and by $2''$ we know that $w^\varphi \in R_+^{1 + n}$.
\end{proof}

\begin{ParSub}\label{par:dyadic}
Letting $d$ be a positive integer, we shall denote by $\ZZ[\frac{1}{d}]$ the subring of $\QQ$ consisting of all rational numbers that can be expressed in the form $\frac{p}{d^n}$ with $p \in \ZZ$ and $n \in \NN$.  Equivalently, $\ZZ[\frac{1}{d}]$ is the localization of $\ZZ$ at the element $d \in \ZZ$, i.e., the localization of $\ZZ$ with respect to the multiplicative subset $\{d^n \mid n \in \NN\} \subseteq \ZZ$.  We shall call $\ZZ[\frac{1}{d}]$ the ring of \textbf{$d$-adic fractions}.  The ring $\ZZ[\frac{1}{2}]$ is usually called the ring of \textbf{dyadic rationals}\footnote{However, the term \textit{$p$-adic rational} for a prime $p$ is often employed in other senses, in connection with the \textit{$p$-adic numbers}.  For this reason, we have chosen to use a different name for the localization $\ZZ[\frac{1}{d}]$.}.  Under the natural order that $\ZZ[\frac{1}{d}]$ inherits from $\QQ$, $\ZZ[\frac{1}{d}]$ is a preordered ring.
\end{ParSub}

\begin{ParSub}\label{par:uniq_mor_from_z}
For any preordered ring $R$, there is a unique morphism of preordered rings $e:\ZZ \rightarrow R$, and for each element $n \in \ZZ$ we denote the associated element $e(n)$ of $R$ by $n$, in accordance with the usual abuse of notation.
\end{ParSub}

\begin{PropSub}\label{thm:preord_dalg}
For a preordered ring $R$ and a positive integer $d$, the following are equivalent:
\begin{enumerate}
\item The element $d$ of $R$ is invertible and its inverse lies in $R_+$.
\item There is a unique morphism of preordered rings $e^\sharp:\ZZ[\frac{1}{d}] \rightarrow R$.
\end{enumerate}
\end{PropSub}
\begin{proof}
The implication 2 $\Rightarrow$ 1 is immediate since $d$ is invertible in $\ZZ[\frac{1}{d}]$ and its inverse lies in $\ZZ[\frac{1}{d}]_+$.  Conversely if 1 holds then by using the universal property of the localization $\ZZ[\frac{1}{d}]$ of $\ZZ$ and the fact that $\ZZ$ is an initial object of the category of rings, we find that there is a unique ring homomorphism $e^\sharp:\ZZ[\frac{1}{d}] \rightarrow R$.  Further, $e^\sharp$ is monotone since if $\frac{p}{d^n} \in \ZZ[\frac{1}{d}]_+$ with $p \in \ZZ$ then $p \in \ZZ_+$ and hence $e^\sharp(\frac{p}{d^n}) = p \cdot (d^{-1})^n \in R_+$ since $p,d^{-1} \in R_+$.
\end{proof}

\begin{DefSub}
Let $d$ be a positive integer.  A preordered ring $R$ is said to be a \textbf{preordered algebra over the $d$-adic fractions}, or a \textbf{preordered} $\ZZ[\frac{1}{d}]$-\textbf{algebra}, if $R$ satisfies the equivalent conditions of \bref{thm:preord_dalg}.  Note that since $e:\ZZ \rightarrow R$ maps $\ZZ$ into the centre of $R$, it follows that $e^\sharp:\ZZ[\frac{1}{d}] \rightarrow R$ maps $\ZZ[\frac{1}{d}]$ into the centre of $R$.  For each $d$-adic fraction $q \in \ZZ[\frac{1}{d}]$ we write the element $e^\sharp(q)$ of $R$ simply as $q$.
\end{DefSub}

\begin{ExaSub}
If we define a preordered ring $R$ by taking the underlying ring of $R$ to be $\QQ$ but taking $R_+$ to be the subrig $\NN \subseteq R$, then $R$ is not a preordered $\ZZ[\frac{1}{2}]$-algebra in the above sense.
\end{ExaSub}

\begin{DefSub}
An element $u$ of a preordered commutative monoid $(M,+,0)$ is an \textbf{order unit} for $M$ if for every $x \in M$ there exists some $n \in \NN$ such that $x \lt nu$.  An element $u$ of a preordered ring $R$ is said to be an \textbf{order unit for the positive part} of $R$ if $u$ is an order unit for the additive monoid of $R_+$.
\end{DefSub}

\begin{ThmSub}\label{thm:suff_conds_for_aff_ext_prop}
Let $R$ be a preordered algebra over the $d$-adic fractions, for some integer $d > 1$, and suppose that $1$ is an order unit for the positive part of $R$.  Then $R$ has the affine extension property.
\end{ThmSub}
\begin{proof}
We need to show that if $\varphi:R^n_+ \rightarrow R_+$ is a left $R_+$-affine map then 
\begin{equation}\label{eq:needed_eqn}\varphi(x) = \varphi(0) + \sum_{i = 1}^n x_i(\varphi(b_i) - \varphi(0))\end{equation}
in $R$ for all $x = (x_1,...,x_n) \in R^n_+$, where $b_i \in R^n_+$ is the $i$-th standard basis vector.  Let us first treat the case where $n = 1$, so that $\varphi:R_+ \rightarrow R_+$.  Letting $\delta = \varphi(1) - \varphi(0) \in R$, we must show that $\varphi(x) = \varphi(0) + x\delta$ for all $x \in R_+$.  To this end, we shall first prove that for each $m \in \NN$,
\begin{equation}\label{eq:common_diff}\varphi(m + 2) - \varphi(m + 1) = \varphi(m + 1) - \varphi(m)\end{equation}
(with the notational convention of \bref{par:uniq_mor_from_z}).  Indeed, the equation
$$m + 1 \;=\; \frac{1}{d}m + \frac{d - 2}{d}(m + 1) + \frac{1}{d}(m + 2)$$
holds in $R$ and expresses $m + 1$ as a left $R_+$-affine combination of the elements \linebreak $m, m + 1, m + 2$ of $R_+$, so
$$\varphi(m + 1) = \frac{1}{d}\varphi(m) + \frac{d - 2}{d}\varphi(m + 1) + \frac{1}{d}\varphi(m + 2)\;.$$
Multiplying both sides of this equation by $d$, we readily compute that \eqref{eq:common_diff} holds.  By induction on $m \in \NN$ the common difference \eqref{eq:common_diff} is $\delta = \varphi(1) - \varphi(0)$, so by another induction on $m \in \NN$ we find that $\varphi(m) = \varphi(0) + m\delta$ in $R$.  Now for an arbitrary element $x \in R_+$, since 1 is an order unit for $R_+$ there is some $m \in \NN$ such that $x \lt m$ in $R$, and we can take $m$ to be a power of $d$ so that $m$ has an inverse $\frac{1}{m} \in R_+$ since $R$ is a preordered $\ZZ[\frac{1}{d}]$-algebra.  But then $\frac{1}{m}x \lt 1$ in $R$ and hence both $1 - \frac{1}{m}x$ and $\frac{1}{m}x$ lie in $R_+$, so we can express $x$ as a left $R_+$-affine combination 
$$x = (1 - \frac{1}{m}x) \cdot 0 + \frac{1}{m}x \cdot m$$
of the elements $0,m \in R_+$ (using the fact that $\frac{1}{m},m$ lie in the centre of $R$).  Hence since $\varphi$ is left $R_+$-affine we compute that
$$
\begin{array}{ccccc}
\varphi(x) & = & \left(1 - \frac{1}{m}x\right)\varphi(0) + \frac{1}{m}x\varphi(m) & = & \left(1 - \frac{1}{m}x\right)\varphi(0) + \frac{1}{m}x(\varphi(0) + m\delta)\\
           & = & \varphi(0) - \frac{1}{m}x\varphi(0) + \frac{1}{m}x\varphi(0) + \frac{1}{m}xm\delta & = & \varphi(0) + x\delta
\end{array}
$$
since $m,\frac{1}{m}$ lie in the centre of $R$.

Having thus established \eqref{eq:needed_eqn} in the case $n = 1$, we now treat the general case.  Given a left $R_+$-affine map $\varphi:R_+^n \rightarrow R_+$ and an element $x \in R_+^n$, note that the composite map
$$\varphi_x = \left(R_+ \xrightarrow{(-) \cdot x} R^n_+ \xrightarrow{\varphi} R_+\right)$$
is left $R_+$-affine and is given by $\varphi_x(r) = \varphi(rx)$, so by what we have established above we find that
\begin{equation}\label{eq:phi_on_sc_mults_of_x}
\begin{array}{lllll}
\varphi(rx) & = & \varphi_x(r) & = & \varphi_x(0) + r(\varphi_x(1) - \varphi_x(0))\\
            & = & \varphi(0x) + r(\varphi(1x) - \varphi(0x)) & = & \varphi(0) + r(\varphi(x) - \varphi(0))\;.
\end{array}
\end{equation}
for all $r \in R_+$.  Next let $\gamma = \sum_{i = 1}^n x_i$.  Since $\gamma$ lies in $R_+$ and $1$ is an order unit for $R_+$ there is some $m \in \NN$ such that $\gamma \lt m$ in $R$, and we can take $m$ to be a power of $d$ so that $m$ has an inverse $\frac{1}{m} \in R_+$.  Now $\frac{1}{m}\gamma \lt 1$, so $1 - \frac{1}{m}\gamma \gt 0$.  The elements $1 - \frac{1}{m}\gamma,\frac{1}{m}x_1,...,\frac{1}{m}x_n$ of $R_+$ sum to 1, and so we can now express $\frac{1}{m}x$ as a left $R_+$-affine combination
$$\frac{1}{m}x = (1 - \frac{1}{m}\gamma) \cdot 0 + \sum_{i = 1}^n\frac{1}{m}x_ib_i$$
of the elements $0,b_1,...,b_n$ of $R_+^n$.  Hence since $\varphi$ is left $R_+$-affine we find that
$$\varphi(\frac{1}{m}x) = (1 - \frac{1}{m}\gamma) \varphi(0) + \sum_{i = 1}^n\frac{1}{m}x_i\varphi(b_i)\;.$$
But by \eqref{eq:phi_on_sc_mults_of_x} we know that $\varphi(\frac{1}{m}x) = \varphi(0) + \frac{1}{m}(\varphi(x) - \varphi(0))$, so
$$\varphi(0) + \frac{1}{m}(\varphi(x) - \varphi(0)) = (1 - \frac{1}{m}\gamma) \varphi(0) + \sum_{i = 1}^n\frac{1}{m}x_i\varphi(b_i)\;.$$
Multiplying both sides by $m$,
$$m\varphi(0) + (\varphi(x) - \varphi(0)) = (m - \gamma)\varphi(0) + \sum_{i = 1}^n x_i\varphi(b_i)\;,$$
so
$$
\begin{array}{lllll}
\displaystyle{\varphi(x)} & = & \displaystyle{(1 - \gamma)\varphi(0) + \sum_{i = 1}^n x_i\varphi(b_i)} &  & \\
           & = & \displaystyle{\left(1 - \sum_{i = 1}^n x_i\right)\varphi(0) + \sum_{i = 1}^n x_i\varphi(b_i)} & = & \displaystyle{\varphi(0) + \sum_{i = 1}^n x_i(\varphi(b_i) - \varphi(0))}
\end{array}
$$
as needed.
\end{proof}

We now focus on the following class of preordered rings, given by a slight weakening of the notion of \textit{strongly archimedean} partially ordered ring from \cite{Kea}.

\begin{DefSub}\label{def:firmly_arch}
We say that a preordered ring $R$ is \textbf{firmly archimedean} if $R$ is archimedean and $1$ is an order unit for the positive part of $R$.
\end{DefSub}

\begin{RemSub}
Observe that a nonzero totally ordered ring $R$ is firmly archimedean if and only if $R$ is archimedean.  Indeed, in a nonzero archimedean totally ordered ring, $1$ is necessarily an order unit for $(R,+,0)$.
\end{RemSub}

\begin{ThmSub}\label{thm:suff_conds}
Let $R$ be a firmly archimedean preordered algebra over the $d$-adic fractions, for some integer $d > 1$.  Then the Lawvere theory of left $R$-convex spaces (i.e., left $R_+$-affine spaces) and the Lawvere theory of pointed right $R_+$-modules are mutual commutants in the full finitary theory of $R_+$ in $\Set$.
\end{ThmSub}
\begin{proof}
This now follows from \bref{thm:suff_conds_for_aff_ext_prop}, \bref{thm:cvx_cmtnt_charn}, and \bref{thm:taffr_as_cmtnt}.
\end{proof}

\begin{ExaSub}\label{exa:firmly_arch_pralgs_over_d}
Let us fix an integer $d > 1$.
\begin{enumerate}
\item Any subring $R$ of $\RR$ containing the $d$-adic fractions is a firmly archimedean preordered $\ZZ[\frac{1}{d}]$-algebra.  In particular, \bref{thm:suff_conds} applies to both $R = \RR$ and $R = \ZZ[\frac{1}{d}]$.
\item Given a set $X$, the ring $R$ of all bounded real-valued functions on $X$ is a firmly ar\-chi\-me\-de\-an preordered $\ZZ[\frac{1}{d}]$-algebra under the pointwise order.
\item Let $R$ be any subring of the ring of all bounded real-valued functions on a given set $X$, and suppose that $R$ contains all constant functions with values in $\ZZ[\frac{1}{d}]$.  Then $R$ is a firmly archimedean preordered $\ZZ[\frac{1}{d}]$-algebra under the pointwise order.
\item Given a compact topological space, the ring $R = C(X)$ of all continuous real-valued functions on $X$ is a firmly archimedean preordered $\ZZ[\frac{1}{d}]$-algebra.
\end{enumerate}
\end{ExaSub}

\providecommand{\bysame}{\leavevmode\hbox to3em{\hrulefill}\thinspace}
\providecommand{\MR}{\relax\ifhmode\unskip\space\fi MR }
\providecommand{\MRhref}[2]{%
  \href{http://www.ams.org/mathscinet-getitem?mr=#1}{#2}
}
\providecommand{\href}[2]{#2}


\end{document}

%% file: pre.tex
\usepackage{amsmath, amsthm, amssymb, stmaryrd}
\usepackage[all]{xy}
\usepackage{fullpage}
\usepackage{titlesec}
\usepackage{mathrsfs}
\usepackage{amsbsy}
\usepackage{turnstile}
\usepackage{bbm}
\usepackage{yhmath}
\usepackage{tensor}
\usepackage{graphics}
\usepackage{enumitem}
\usepackage{calligra}
\usepackage{verbatim}
\usepackage{array}

\usepackage{setspace}

\usepackage[pdftex,bookmarks=true]{hyperref}
\usepackage[svgnames]{xcolor}
\hypersetup{
    linkbordercolor = {PaleTurquoise},
    citebordercolor = {PaleGreen},
}

\makeatletter
\newcommand*{\@old@slash}{}\let\@old@slash\slash
\def\slash{\relax\ifmmode\delimiter"502F30E\mathopen{}\else\@old@slash\fi}
\makeatother

\titleformat{\section}{\normalsize\bfseries}{\thesection}{1em}{}
\titleformat{\subsection}{\normalsize\bfseries}{\thesubsection}{1em}{}

\numberwithin{equation}{subsection}

\theoremstyle{plain}

\newtheorem{PropSub}[subsection]{Proposition}
\newtheorem{LemSub}[subsection]{Lemma}
\newtheorem{CorSub}[subsection]{Corollary}
\newtheorem{ThmSub}[subsection]{Theorem}

\theoremstyle{definition}

\newtheorem{DefSub}[subsection]{Definition}
\newtheorem{ExaSub}[subsection]{Example}
\newtheorem{RemSub}[subsection]{Remark}
\newtheorem{ParSub}[subsection]{}

\newtheorem{NotnSub}[subsection]{Notation}

\newtheorem*{Acknowledgement}{Acknowledgement}

\newcommand*{\emptybox}{\leavevmode\hbox{}}

\newdir{ >}{{}*!/-10pt/@{>}}

\DeclareMathAlphabet{\mathpzc}{OT1}{pzc}{m}{it}
\DeclareMathAlphabet{\mathcalligra}{T1}{calligra}{m}{n}

\newcommand{\bref}[1]{\textnormal{\ref{#1}}}
\newcommand{\pbref}[1]{\textnormal{(\ref{#1})}}

\newcommand{\A}{\ensuremath{\mathscr{A}}}

\newcommand{\C}{\ensuremath{\mathscr{C}}}

\newcommand{\normalJ}{\ensuremath{\mathscr{J}}}
\newcommand{\J}{\ensuremath{{\kern -0.4ex \mathscr{J}}}}

\newcommand{\JF}{\ensuremath{{\kern -0.4ex \mathscr{J}_{\kern -0.05ex F}}}}

\newcommand{\sP}{\ensuremath{\mathscr{P}}}

\newcommand{\sS}{\ensuremath{\mathscr{S}}}
\newcommand{\T}{\ensuremath{\mathscr{T}}}
\newcommand{\U}{\ensuremath{\mathscr{U}}}
\newcommand{\V}{\ensuremath{\mathscr{V}}}

\newcommand{\CAT}{\ensuremath{\operatorname{\textnormal{\text{CAT}}}}}

\newcommand{\VCAT}{\ensuremath{\V\textnormal{-\text{CAT}}}}

\newcommand{\NN}{\ensuremath{\mathbb{N}}}

\newcommand{\QQ}{\ensuremath{\mathbb{Q}}}

\newcommand{\RR}{\ensuremath{\mathbb{R}}}

\newcommand{\TT}{\ensuremath{\mathbb{T}}}

\newcommand{\ZZ}{\ensuremath{\mathbb{Z}}}

\newcommand{\ob}{\ensuremath{\operatorname{\textnormal{\textsf{ob}}}}}
\newcommand{\mor}{\ensuremath{\operatorname{\textnormal{\textsf{mor}}}}}

\newcommand{\stt}{\mathbin{\tilde{*}}}

\newcommand{\ca}[1]{\scalebox{0.85}{\raisebox{0.3ex}{$|$}} #1\scalebox{0.85}{\raisebox{0.3ex}{$|$}}}

\newcommand{\Th}{\ensuremath{\textnormal{Th}}}

\newcommand{\ThJ}{\ensuremath{\Th_{\kern -0.5ex \normalJ}}}
\newcommand{\SubTh}{\ensuremath{\textnormal{SubTh}}}
\newcommand{\Vfp}{\ensuremath{\V_{\kern -0.5ex fp}}}

\newcommand{\Ev}{\ensuremath{\textnormal{\textsf{Ev}}}}

\newcommand{\inJ}[1]{#1 \in \normalJ\kern -1.2ex}

\newcommand{\VCATJ}{\VCAT_{\kern -0.7ex \normalJ}}
\newcommand{\PhiJ}{\Phi_{\kern -0.8ex \normalJ}}
\newcommand{\MndJ}{\Mnd_{\kern -0.8ex \normalJ}}

\newcommand{\Set}{\ensuremath{\operatorname{\textnormal{\text{Set}}}}}
\newcommand{\FinCard}{\ensuremath{\operatorname{\textnormal{\text{FinCard}}}}}

\newcommand{\Ord}{\ensuremath{\operatorname{\textnormal{\text{Ord}}}}}
\newcommand{\Bool}{\ensuremath{\operatorname{\textnormal{\text{Bool}}}}}
\newcommand{\SLat}{\ensuremath{\operatorname{\textnormal{\text{SLat}}}}}
\newcommand{\USLat}{\ensuremath{\operatorname{\textnormal{\text{USLat}}}}}

\newcommand{\Ab}{\ensuremath{\operatorname{\textnormal{\text{Ab}}}}}
\newcommand{\Ring}{\ensuremath{\operatorname{\textnormal{\text{Ring}}}}}

\newcommand{\Mod}[1]{\ensuremath{#1\textnormal{-\text{Mod}}}}
\newcommand{\Aff}[1]{\ensuremath{#1\textnormal{-\text{Aff}}}}
\newcommand{\Cvx}[1]{\ensuremath{#1\textnormal{-\text{Cvx}}}}

\newcommand{\Mat}{\ensuremath{\textnormal{\text{Mat}}}}

\newcommand{\Mnd}{\ensuremath{\operatorname{\textnormal{\text{Mnd}}}}}

\newcommand{\CMon}{\ensuremath{\operatorname{\textnormal{\text{CMon}}}}}

\newcommand{\Alg}[1]{\ensuremath{#1\kern -.5ex\operatorname{\textnormal{-\text{Alg}}}}}
\newcommand{\CAlg}[1]{\ensuremath{#1\kern -.5ex\operatorname{\textnormal{-\text{CAlg}}}}}

\newcommand{\CRProf}{\ensuremath{\operatorname{\textnormal{\text{CRProf}}}}}
\newcommand{\CRProfJ}{\ensuremath{\CRProf_{\kern -0.6ex \normalJ}}}

\newcommand{\Adjn}[6]{\xymatrix {#1 \ar@/_0.5pc/[rr]_{#2}^(0.4){#4}^(0.6){#5}^{\top} & & #6 \ar@/_0.5pc/[ll]_{#3}}}
\newcommand{\Equiv}[6]{\xymatrix {#1 \ar@/_0.5pc/[rr]_{#2}^(0.4){#4}^(0.6){#5}^{\sim} & & #6 \ar@/_0.5pc/[ll]_{#3}}}

\newcommand{\lt}{\leqslant}
\newcommand{\gt}{\geqslant}

\newcommand{\op}{\ensuremath{\textnormal{op}}}

\newcommand{\aff}{\ensuremath{\textnormal{aff}}}

\newcommand{\pushoutcorner}{\ar@{}[dr]|(.3)\ulcorner}
\newcommand{\pullbackcorner}{\ar@{}[dr]|(.3)\lrcorner}

\newcommand{\cmt}[1]{}